\documentclass[reqno,12pt]{amsart}
\usepackage{amssymb,amsmath}
\usepackage[dvips]{color}
\usepackage[dvips,final]{graphics}
\usepackage{epstopdf}
\usepackage{amscd}
\usepackage[latin1]{inputenc}

\usepackage{graphicx}
\usepackage{xcolor}
\usepackage{tikz}

\theoremstyle{plain}
\newtheorem{teo}{Theorem}[section]      
\newtheorem{prop}[teo]{Proposition}    
\newtheorem{lem}[teo]{Lemma}            

\theoremstyle{definition}               
\theoremstyle{remark}                   
\newtheorem{oss}[teo]{Remark}
\newcommand{\re}{\mathbb{R}}
\newcommand{\N}{\mathbb{N}}

\numberwithin{equation}{section}

\begin{document}

\title{Interpolation inequalities in pattern formation}
\author[Eleonora Cinti]{Eleonora Cinti}
\address{E.C., Wierstrass Institute for Applied Analysis and Stochastics,
Mohrenstr. 39,
10117 Berlin (Germany) }

\email{cinti@wias-berlin.de}
\author[Felix Otto]{Felix Otto}

\address{F.O., Max Planck Institute for Mathematics in the Sciences\\
Inselstr. 22\\
04103 Leipzig (Germany)}  \email{otto@mis.mpg.de}
\thanks{
The first author was supported by grants MTM2011-27739-C04-01 (Spain), 2009SGR345 (Catalunya) and by the ERC Starting Grants ``AnOptSetCon'' n. 258685 and ``Epsilon'' n. 277749. 
}
\subjclass[2010]{49J40, 47J20}
\keywords{Interpolation inequalities, optimal transport, pattern formation.}
\begin{abstract}
We prove some interpolation inequalities which arise in the analysis of pattern formation in physics. They are the strong version of some already known estimates in weak form
that are used to give a lower bound of the energy in many contexts (coarsening and branching in micromagnetics and superconductors).
The main ingredient in the proof of our inequalities is a geometric construction which was first used by Choksi, Conti, Kohn, and one of the authors in \cite{CCKO} in the study of
branching in superconductors.
\end{abstract}
\maketitle
\section{Introduction}
In this paper we establish some interpolation inequalities which are connected with the study of certain physical phenomena. More precisely, the inequalities that we prove are the strong versions of some already known interpolation estimates in weak form that play a crucial role in the study of pattern formation in physics.

In many physical phenomena described by a variational model, in order to understand why certain patterns are observed, it is natural to study the features of pattern with close to minimal energy.
This requires a good understanding of at least the scaling of the minimal energy in terms of the model parameters. Upper bounds on the minimal energy are obtained by physically motivated trial patterns (Ansatz). Typically, Ansatz-free lower bounds rely on suitable interpolation inequalities that involve some functional norms related to the energy under consideration.

The first two  interpolation inequalities that we present here involve the BV-norm and the $\dot{H}^{-1}$-norm of a function $u$. The first estimate holds in any dimension $d$ and it
is stated in Proposition \ref{P1} below. The second inequality (see Proposition \ref{P2}) holds in dimension $2$ for functions bounded below, and it
improves the result in Proposition \ref{P1} by a logarithmic factor.

Here with $\dot{H}^{-1}$ we refer to the homogeneous negative Sobolev space $H^{-1}$.
Let $u$ be a periodic function defined on the torus $[0,\Lambda]^d$ and with vanishing average.
We recall that the $\dot{H}^{-1}$-norm of $u$ is defined as follows
\begin{eqnarray*}
 \|u\|_{\dot{H}^{-1}}=\| |\nabla|^{-1}u\|_2^2&:=&\inf_j\left\{\int_{[0,\Lambda]^d} |j|^2\:\big|\:j\;\mbox{is periodic and}\;\nabla \cdot j=u\right\}\\
&=& \int_{[0,\Lambda]^d} |\nabla \varphi|^2\quad \mbox{where}\;\;-\Delta \varphi =u.
\end{eqnarray*}
It can also be defined via Fourier transform:
 $$ \| |\nabla|^{-1}u\|_2^2=\int (|k|^{-1}|F(u)|)^2 dk,$$
 where $\int dk$ has to be interpreted in a discrete sense.

\begin{prop}\label{P1}
Let $u:[0,\Lambda]^d\rightarrow \re$ be a periodic function such that  $\int_{[0,\Lambda]^d} u=0$.

Then, there exists a constant $C>0$ only depending on $d$ such that
\begin{equation}\label{est1}
\|u\|_{\frac{4}{3}}\;\le\;C \|\nabla u\|_{1}^{\frac{1}{2}}\||\nabla|^{-1}u\|_{2}^{\frac{1}{2}}.
\end{equation}
\end{prop}
Let us comment on the nature of \eqref{est1}, which is the model of three further interpolation estimates appearing in this paper. Its interest for applications is that it relates:
\begin{itemize}
\item the $L^1$-norm of $\nabla u$, with its obvious geometric interpretation via the coarea formula and its ensuing connection to an interfacial energy;
\item the $L^2$-norm of $\nabla \varphi$ with $-\Delta \varphi=u$: $\int |\nabla \varphi|^2$ is a prototype for a field energy, e.g. the electrostatic energy of $u$ interpreted as a charge distribution;
\item some $L^p$-norm of $u$ itself, which typically can be estimated below using the non convex features of the model where $u$ plays the role of an order parameter.
\end{itemize}

Let us go one step deeper: \eqref{est1} is one end-point estimate in the one-parameter family of Gagliardo-Nirenberg estimates
\begin{equation}\label{GN}\|u\|_p\leq C\|\nabla u\|_q^{\frac{1}{2}}\| |\nabla|^{-1}u\|_2^{\frac{1}{2}},\quad 1\leq q\leq \infty,
\end{equation}
where the integrability exponent $p$ on the left-hand side is determined by the relation
$$\frac{1}{p}=\frac{1}{2}\cdot\frac{1}{q}+\frac{1}{2}\cdot\frac{1}{2},$$
that is $p=\frac{4q}{2+q}$. This exponent is the only one compatible with rescaling of length. At the same time, the estimates \eqref{GN} are ``extensive'', that is, compatible with taking volume averages, i.e. passing from $\int_{[0,\Lambda]^d}dx$ to $\Lambda^{-d}\int_{[0,\Lambda]^d}dx$, as can be seen from rewriting \eqref{GN} as
$$
\left(\int_{[0,\Lambda]^d} |u|^pdx\right)^{\frac{1}{p}}\leq C \left(\int_{[0,\Lambda]^d} |\nabla u|^q\right)^{\frac{1}{2}\cdot\frac{1}{q}}\left(\int_{[0,\Lambda]^d} | |\nabla|^{-1}u|^2\right)^{\frac{1}{2}\cdot\frac{1}{2}}.
$$
This extensivity is a crucial feature in the applications: the estimate should be compatible with taking volume averages, since it should be oblivious to the artificially introduced period $\Lambda$.
In this sense, the family of estimates \eqref{GN} is ``orthogonal''  to Sobolev estimates, which are saturated by localized functions.

Among the family of estimates \eqref{GN}, the one with $q=2$ (and thus $p=2$) is particularly simple and can be established by Fourier representation. The estimates for $1<p<\infty$ are fairly easy to prove using Calderon-Zygmund theory, i.e. the maximal regularity for $-\Delta$ in $L^q$-spaces. The latter is no longer available for the end-point $q=1$ and thus it is not surprising that this estimate was only recently established. It was first proved by Cohen, Dahmen, Daubechies, and DeVore (see Theorem 1.5 in \cite{CDDD}) using the wavelet analysis of the space BV, while our proof uses a technique introduced by Ledoux (see Theorem 1 in \cite{Led}) to give a direct proof of some improved Sobolev inequalities.

The following proposition improves the result in Proposition \eqref{P1} by a factor $\ln^{\frac{1}{4}}u$ and it holds in dimension $d=2$ for functions bounded below.

\begin{prop}\label{P2}
Let $u:[0,\Lambda]^2\rightarrow \re$ be a periodic function such that $u\geq -1$ and $\int_{[0,\Lambda]^d} u=0$.

Then, there exists a constant $C>0$ such that
\begin{equation}\label{inter-est}
\|u\ln^{\frac{1}{4}}\max\{u,e\}\|_{\frac{4}{3}}
\;\leq\;C\|\nabla u\|_1^\frac{1}{2}\;\||\nabla|^{-1}u\|_2^\frac{1}{2}.\end{equation}
\end{prop}
The main ingredient in the proof of Proposition \ref{P2} is a geometric construction (see Lemma \ref{geom}), which was first used by Choksi, Conti, Kohn, and  one of the authors in \cite{CCKO} in the context of branched patterns in superconductors.

These two interpolation inequalities are naturally connected to many physical problems (coarsening, domain branching in ferromagnets, superconductors,
twin branching in shape memory alloys) whose energy is given by the competition of two main terms: an interfacial energy (described by a BV-norm) and a field energy (described by an $\dot{H}^{-1}$-norm ).
In Section 2 we will explain how estimates \eqref{est1} and \eqref{inter-est} are used to give lower bound for the energy in two different problems: coarsening and micromagnetics.

Our last two interpolation inequalities are connected with
 the study of branching in superconductors and they both involve the Wasserstein distance.

We recall (see for instance \cite{Vil}) that
the Wasserstein distance between two measures with densities $u$ and $v$ is given by
\begin{equation}\label{was}W_2^2(u,v):=\inf\left\{\int \int |x-y|^2 d\pi(x,y)|\int d\pi(\cdot, y)=u,\;\int d\pi(x,\cdot)=v\right\}.\end{equation}

The measure on the product space $\pi$ is called \textit{transportation plan} and it is admissible if its projections to first and second coordinates are measures
with densities $u$ and $v$ respectively.


A useful property of the Wasserstein distance is the \textit{Kantorovich duality} (see Chapter 5 in \cite{Vil}), which allows to write the Wasserstein distance in a dual form in the following way:
\begin{equation}\label{Kan}W_2^2(u,v)=\sup\left\{ \int u(x)\varphi(x) dx + \int v(y)\psi(y)dy\:|\: \varphi(x)+\psi(y)\leq |x-y|^2\right\}.\end{equation}
This property will be useful in the proof of Propositions \ref{strong-W} and \ref{cor} below.

Loosely speaking, the Wasserstein distance between $u$ and $v$ is a nonlinear version of a negative norm of $u-v$. The first proposition states an interpolation estimate with the $\dot H^{-1}$-norm of $u$ (or rather of $u-1$) replaced by $W_2(u,1)$.
\begin{prop}\label{strong-W}
 Let $u:[0,\Lambda]^d\rightarrow \re$ be a periodic function such that $u\geq 0$ and $\Lambda^{-d}\int_{[0,\Lambda]^d} u=1$.

Then, there exists a constant $C>0$ only depending on $d$ such that
\begin{equation}\label{strong-W-eq}
 \|(u-C)_+\|_{\frac{2+3d}{3d}}\leq C\left(\|\nabla u\|_1\right)^{\frac{2d}{2+3d}} (W_2^2(u,1))^{\frac{d}{2+3d}},
\end{equation}
where $(u-C)_+:=\max\{u-C,0\}$.
\end{prop}
In the recent paper \cite{Led2}, M. Ledoux showed that our inequality \eqref{strong-W-eq} is actually a particular case in a all family of
interpolation inequalities (called Sobolev-Kantorovich inequalities) that  hold in the more general setting of non-negatively curved (weighted) Riemannian manifolds. The proofs in \cite{Led2} rely on the  use of heat flows and Harnack inequalities.

Let us comment a bit on  interpolation estimate \eqref{strong-W-eq}. Because of the formal scalings
\begin{equation*}
\begin{split}
\|\nabla u\|_1&\;\;\mbox{has unit of}\;\;\mbox{length}^{-1}\times \mbox{volume},\\
W_2^2(u,1)&\;\;\mbox{has unit of}\;\;\mbox{length}^{2}\times \mbox{volume},\\
\|(u-C)_+\|_p&\;\;\mbox{has unit of}\;\;\hspace{4em} \mbox{volume}^{\frac{1}{p}},\end{split}\end{equation*}
the pair of exponents on the right-hand side of \eqref{strong-W-eq} is determined by the integrability exponent $p=\frac{2+3d}{3d}$ appearing on the left-hand side. The pair of exponents is the only one compatible with rescaling of length and taking the volume average. The exponent $p$ could not be inferred by a simple scaling argument. It is the exponent that appears on the left-hand side in the \emph{linear} interpolation estimate
$$\|u\|_{\frac{2+3d}{3d}}\leq C\|\nabla u\|_1^{\frac{2d}{2+3d}} \| |\nabla|^{-\frac{2d}{d+2}} u \|_{\frac{d+2}{d}}^{\frac{d+2}{2+3d}}.$$
In this sense, $W_2^2(u,1)$, at least as seen from the ``peaks'' of $u$ (i.e. the part of $u$ that is much larger than $1$), behaves as the negative fractional Sobolev norm
$$\| |\nabla|^{-\frac{2d}{d+2}} u \|_{\frac{d+2}{d}}^{\frac{d+2}{d}}=\int | |\nabla|^{-\frac{2d}{d+2}}u|^{\frac{d+2}{d}}.
$$

%

%
The last nonlinear interpolation estimate of this paper replaces the negative norm in Proposition \ref{P1} and the Wasserstein distance in Proposition \ref{P2} by a \emph{mixture} of both. We recall that by the $K$-method of interpolation spaces (see Chapter 3 in \cite{BL}) a norm $\|\cdot\|$ intermediate between some norms $\|\cdot\|_{(0)}$ and $\|\cdot\|_{(1)}$ can be obtained via the construction
$$
\|u\|=\sup_{\mu>0} \inf_{v:\re^d\rightarrow \re}\left\{ \mu \|u-v\|_{(0)}+ \mu^{-\frac{1-\theta}{\theta}}\|v\|_{(1)} \right\}.$$
By the equivalence of approximation and interpolation theory (see Chapter 7 in \cite{BL}) in term of
$$\|u\|\sim \left(\sup_{M>0} M^{\frac{\theta}{1-\theta}} \inf_{v:\re^d\rightarrow \re}\left\{  \|u-v\|_{(0)}\:\big|\: \|v\|_{(1)}\leq M\right\}\right)^{1-\theta},$$
we also obtain a representation with mixed homogeneity
$$
\|u\|\sim \left(\sup_{\mu>0} \inf_{v:\re^d\rightarrow \re}\left\{ \mu \|u-v\|_{(0)}+ \left(\mu^{-\frac{1-\theta}{\theta}}\|v\|_{(1)}\right)^2 \right\}\right)^{\frac{2-\theta}{2}}.
$$

In Proposition \ref{cor}, with an application described in the next section in mind, we replace $\|u-v\|_{(0)}$ by $W_2^2(u,v)$ and $\|v\|_{(1)}$ by the homogeneous fractional Sobolev norm $\| |\nabla|^{-\frac{1}{2}} v\|_2$, and choose $\theta=\frac{4}{d+3}$. This leads to
\begin{eqnarray}\label{last1}
&\displaystyle \sup_{\mu>0}\inf_{v:{\re^d}\rightarrow \re}\left\{\mu W_2^2(u,v)+\left(\mu^{-\frac{1}{4}(d-1)} \| |\nabla|^{-\frac{1}{2}} v\|_2\right)^2\right\}^{\frac{d+1}{d+3}}\\
&\hspace{1em}=\displaystyle \sup_{\nu>0}\inf_{v:{\re^d}\rightarrow \re}\left\{\nu^{\frac{2}{d+1}} W_2^2(u,v)+\nu^{-\frac{d-1}{d+1}} \| |\nabla|^{-\frac{1}{2}} v\|_2^2\right\}^{\frac{d+1}{d+3}}.\nonumber
\end{eqnarray}
%
\begin{prop}\label{cor}
Let $u:\re^d\rightarrow \re$ be such that $u \geq 0$ and $\int u< \infty$.

Then there exists a constant $C>0$ only depending on $d$ such that

\[ \label{ineq-H-sup}
\|u\|_{\frac{3d+3}{3d+1}}\leq C \|\nabla u\|_1^{\frac{2d}{3d+3}}
\cdot \sup_{\nu>0}\inf_{v:{\re^d}\rightarrow \re}\left\{\nu^{\frac{2}{d+1}}W_2^2(u,v)+\nu^{-\frac{d-1}{d+1}}\| |\nabla|^{-\frac{1}{2}}v\|^2_2\right\}^{\frac{1}{3}}.
\]
\end{prop}
In our proof, we will use the following representation of the fractional Sobolev norms
\begin{equation}\label{H-norm}\| |\nabla|^{-\frac{1}{2}} f\|_2:=\int_{\re^d}\int_{\re^d} \frac{|f(x)-f(\overline x)|^2}{|x-\overline x|^{d-1}}dx d\overline x.\end{equation}

Let us comment on Proposition \ref{cor}. We claim that it can be interpreted in the sense of ``$W_2^2(u,v)$ acts as $\| |\nabla|^{-2}(u-v)\|_1$''. Indeed, if in \eqref{last1} we replace $W_2^2(u,v)$ by  $\| |\nabla|^{-2}(u-v)\|_1$, we have
$$
\sup_{\mu>0}\inf_{v:\re^d\rightarrow \re}\left\{\mu \| |\nabla|^{-2} (u-v)\|_1 +\left(\mu^{-\frac{1}{4}(d-1)} \| |\nabla|^{-\frac{1}{2}}v\|_2\right)^2\right\}^{\frac{d+1}{d+3}}.
$$
According to the above discussion, this expression is equivalent to
$$\sup_{\mu>0}\inf_{v:\re^d\rightarrow \re}\left\{\mu \| |\nabla|^{-2} (u-v)\|_1 +\mu^{-\frac{1-\theta}{\theta}} \| |\nabla|^{-\frac{1}{2}}v\|_2\right\},
$$
where $\theta=\frac{4}{d+3}$. The norm of this interpolation space $[\dot H_1^{-2},\dot H_2^{-\frac{1}{2}}]_\theta$ has the same scaling as the norm of the fractional Sobolev space $\dot H^{(1-\theta)(-2)+\theta\left(-\frac{1}{2}\right)}_{\frac{1}{(1-\theta)1+\theta\frac{1}{2}}}=\dot H^{-\frac{2d}{d+3}}_{\frac{d+3}{d+1}}$, that is
$$
\| |\nabla|^{-\frac{2d}{d+3}}u\|_{\frac{d+3}{d+1}}.$$
With this substitution, the nonlinear interpolation estimate of Proposition \ref{cor} turns into the much more standard linear interpolation estimate
$$
\|u\|_{\frac{3d+3}{3d+1}}\leq C\|\nabla u\|_1^{\frac{2d}{3d+3}} \| |\nabla|^{-\frac{2d}{d+3}} u\| _{\frac{d+3}{d+1}}^{\frac{d+3}{3d+3}}.
$$
In this sense, $W_2^2(u,v)$ acts as $\| |\nabla |^{-2}(u-v)\|_1$.

When deriving an Ansatz-free lower bound in the physical applications, only a weak version of these four interpolation inequalities is needed, in the sense that it is enough to have these estimates with the
$L^p$-norms on the left-hand side replaced by the weak $L^p_w$-norm. The contribution of this paper is to pass from the weak formulation to the strong one.

The paper is organized as follows:
\begin{itemize}	
\item In Section 2 we expose the connection between the above interpolation inequalities and the study of pattern formation by considering three examples: coarsening, micromagnetics, and superconductors.
\item In Section 3 we prove Proposition \ref{P1}, using a method introduced by Ledoux in \cite{Led}.
\item In Section 4 we give the proof of Proposition \ref{P2}, whose main ingredient is the geometric construction given in Lemma \ref{geom}, introduced in \cite{CCKO}.
\item In Section 5 we prove Propositions \ref{strong-W} and \ref{cor}.
\end{itemize}


\section{Three models: coarsening, micromagnetics, and superconductors}

In this section we expose the connection between our interpolation inequalities and three different physical phenomena.

The first model that we consider describes thermodynamically driven demixing, the second one describes the magnetization in a ferromagnet. In the first model, we are interested in the phenomenum of coarsening of the spatial phase distribution, in the second model, we are interested in the phenomenum of domain branching in a strongly uniaxial ferromagnet towards the sample surface.

In both of them a crucial ingredient in the proof of a lower bound for the energy is a weak version of our first two interpolation inequalities (Propositions \ref{P1} and \ref{P2}).

The last model that we consider describes type-I superconductors. Also here, the phenomenum of interest is branching: it is the magnetic flux carrying normal phase that branches when approaching the sample surface. In this phenomenum, several different regimes occurs and Propositions \ref{strong-W} and \ref{cor} are related to the analysis of two of them.

The discussion in this section is not rigorous; the main aim is to show how the quantities involved in our interpolation estimates arise naturally from the energy functionals associated to these three models.

\textit{1. Coarsening}.
The Cahn-Hilliard equation models the thermodynamically driven demixing of a two-component system. The relative concentration is described by an order parameter, say on a large torus, $u:[0,\Lambda]^d\rightarrow \re$. The driving free energy is of Ginzburg-Landau form
\begin{equation}\label{energy}
\mathcal E(u) =\frac{1}{2}\int_{[0,\Lambda]^d}(|\nabla u|^2+  (1-u^2)^2)dx.
\end{equation}
The postulate that the free energy decreases while the order  parameter is conserved in time (i.e. $\frac{d}{dt}\int_{[0,\Lambda]^d}u=0$) is satisfied by the evolution equation
\begin{equation}\label{CH}
\partial_t u-\Delta \frac{\partial \mathcal E}{\partial u}=\partial_t u-\Delta(\Delta u+u-u^3)=0,\end{equation}
which is the Cahn-Hilliard equation.

Starting from a small perturbation of the critical mixture $u \equiv 0$, numerical simulations show after an initial stage the emergency of two convoluted regions in which $u$ is very close to either its equilibrium value $1$ and $-1$, respectively. These regions are separated by a shortly modulated characteristic transition layer given by the one-dimensional minimizer  of \eqref{energy}. In the sequel, the evolution is essentially geometrical (named after Mullin and Sekerka) and is driven by a reduction of the interfacial energy. This leads to a coarsening of the two complementary domains, i.e. to an increase of their average length (as embodied by its average width or radius of curvature). Eventually, this coarsening stops because the average length-scale reaches the artificial period $\Lambda$. The coarsening is observed to be ``statistically'' self-similar in the sense that e.g. the two-point correlation function is (approximately) self-similar. It is characterized by the exponent $1/3$ in the sense that average length scale as $t^{1/3}$.

In \cite{KO}, Kohn and one of the authors proved quantitatively that, independent of the initial data provided that they are well-mixed, coarsening cannot proceed at a faster rate. The strategy for the proof has since been applied to many models that feature coarsening (see e.g. the discussion in \cite{BOS}) and relies on the gradient flow structure. Morally speaking, it consists in converting information on the steepness of the energy landscape (how quickly does the energy decrease as a function of the intrinsic distance to the well-mixed state) into information on the speed of relaxation (how fast does the energy decrease as a function of time), see the discussion in \cite{ORS}: if the landscape is not too steep, the relaxation is not too fast.

As can be easily inferred from \eqref{CH} the Cahn-Hilliard evolution is the gradient flow of the Ginzburg-Landau energy $\mathcal E$ with respect to the Euclidean geometry given by the inner product in $\dot H^{-1}$. Hence showing that the energy landscape is not too steep means establishing an estimate of the form
\begin{equation}\label{coar-ineq}
\mathcal E(u)\| |\nabla|^{-1}u\|_2 \geq \frac{1}{C}\Lambda^{d\left(1+\frac{1}{2}\right)},
\end{equation}
the power  of the system volume $\Lambda^d$ being dictated by the need to take volume averages.

We now heuristically point out the connection to the interpolation estimate from Proposition \ref{P1}. As mentioned above, the evolution relevant for coarsening is well approximated by the free-boundary problem named after Mullin and Sekerka, which is the gradient flow of
$$
\mathcal E_{MS}(u)=\begin{cases} C_0 \int |\nabla u| &\mbox{if}\;\;u\in \{-1,1\}\\
+\infty & else\end{cases}$$
with respect to $\dot H^{-1}$. Hence replacing $\mathcal E$ by $\mathcal E_{MS}$ in \eqref{coar-ineq} leads to
$$\|\nabla u\|_1 \| |\nabla|^{-1}u\|_2 \geq \frac{1}{C}\Lambda^{\frac{3}{2}} \stackrel{u\in\{-1,1\}}=\frac{1}{C}\|u\|_{\frac{4}{3}},$$
which is the squared version of our interpolation estimate.

Let us now turn to Proposition \ref{P2}. If the mixture is strongly off-critical, by which one understands that
$$\Lambda^{-d}\int _{[0,\Lambda]^d} u = -1+\phi \quad \mbox{with}\;\;\phi \ll 1,$$
the above strategy gives the correct scaling of the prefactor in the upper coarsening bound in the volume fraction $\phi$ in dimensions $d>2$. In low volume fraction, numerical simulations show that the region covered by the minority phase $u=1$ breaks up into many connected components which quickly relax to have approximately round shape. Coarsening then proceed by Ostwald ripening: the large balls grow at the expense of the small ones that eventually collapse.

Asymptotic analysis for $d=2$ shows that there is a logarithmic correction to the prefactor in the coarsening rate \cite{NO}. It is the interpolation estimate in Proposition \ref{P2} that correctly captures this logarithmic correction, as shown in \cite{CNO}, where also the weak form of \eqref{inter-est} was established.

\textit{2. Branching in Ferromagnets}.
We are interested  in a ferromagnet in form of a slab of thickness $2t$; for simplicity, we assume periodic boundary conditions in the horizontal directions so that the fundamental domain is given by $[0,\Lambda]^2\times (-t,t)$. The magnetization $m$ is a unit-length vector field $m:[0,\Lambda]^2\times (-t,t)\rightarrow \mathbb S^2$ that locally indicates the mean direction of the spin.

The ground state minimizes the Landau-Lifshitz energy functional (as modified by Brown):
\begin{equation}\label{E-m}
E(m)=d^2 \int_{[0,\Lambda]^2\times(-t,t)} |\nabla m|^2 dx + Q\int_{[0,\Lambda]^2\times(-t,t)} (m_1^2+m_2^2)dx +\int_{[0,\Lambda]^2\times \re} ||\nabla|^{-1} \nabla \cdot m|^2 dx.\end{equation}
The first term, the exchange energy, is of quantuum mechanical origin, and models a short range attraction between the spins and comes with its intrinsic length scale $d$. The second term, the anisotropy energy, comes from the interaction of the magnetization with the lattice structure of the metal; here, it is uniaxial and favors the third axes and thus the values $m=\pm(0,0,1)$. The last term, the stray field energy, can also be written directly in terms of the stray field $h$, which is determined by the static Maxwell equations:
\begin{equation}\label{stray-field}
\int_{[0,\Lambda]^2\times \re} |h|^2 dx,\quad \mbox{where}\;\;\nabla \cdot (h+m)=0,\;\;\nabla \times h=0.
\end{equation}
Both equations are to be interpreted distributionally in $\re^3$, where $m$ is extended trivially outside the slab. In other words, $h$ is the Helmoltz projection of the extended $m$. There is a mathematical analogy from magnetostatics convenient for the intuition: there are two sources for the stray field, namely "volume charges " coming from $-\nabla \cdot m$ in the bulk $\re^2 \times (-t,t)$ of the sample and "surface charges" $\pm m_3$ at the surface $\re^2\times \{-t,t\}$ of the sample. Like in electrostatics, both give rise to the potential field $h$.

Following Hubert (see Chapter 3 in \cite{HS}), let us euristically explain why branching occurs; see also \cite{DKMO,OV}. A constant magnetization in direction of the easy axes, say $m=(0,0,1)$, would have zero exchange and anisotropy energy. However, by \eqref{stray-field}, we would have $h=m$ so that the stray field energy per area of cross section would be $2t$. However, the $\dot{H}^{-1}$-norm of distributional divergence $\nabla \cdot m$, which here comes from the $m_3$-component at the sample surfaces $[0,\Lambda]^2\times \{-t,t\}$ (that is the surface charges), can be reduced by horizontally alternating between $m=(0,0,1)$ and $m=(0,0,-1)$, and thus alternating the sign of the charges. In fact, for magnetization that only depends on $x'=(x_1,x_2)$ and alternates with a period $\omega \ll t$, the stray field energy would behave as the squared $\dot H^{-\frac{1}{2}}$-norm in the horizontal variables $x'=(x_1,x_2)$ of the surface charges $m_3$, that is like $\int_{[0,\Lambda]^2}||\nabla'|^{-\frac{1}{2}} m_3|^2 dx'$.
The exchange energy prevents such jumps of $m$, but they can be replaced by smooth $x_1$-dependent transition layers (Bloch walls) between the magnetizations $(0,0,-1)$ and $(0,0,1)$. These transition layers are determined by a balance between exchange and anisotropy energy in form of
\begin{eqnarray*}
&&d^2\int \left|\frac{dm}{dx_1}\right|^2 dx_1 + Q\int (m_1^2+m_2^2)dx_1\\
&&\hspace{1em} =d^2\int \frac{1}{1-m_3^2}\left(\frac{d m_3}{d x_1}\right)^2 dx_1 + Q\int (1-m_3^2) dx_1\\
&&\hspace{1em}\geq \frac{1}{2} d Q^{\frac{1}{2}} \int \frac{d m_3}{d x_1} dx_1 = d Q^{\frac{1}{2}}.
\end{eqnarray*}
Dimensional analysis shows that the width $\omega_{wall}$ of the transition layers scales as $\omega_{wall}\sim dQ^{-\frac{1}{2}}$ and the above inequality shows that the wall energy per area in vertical direction is given by $dQ^{\frac{1}{2}}$.
This suggests to consider the reduced energy functional
\begin{equation}\label{reduced-E}
E(m)= d Q^{\frac{1}{2}} \int_{[0,\Lambda]^2\times (-t,t)} |\nabla m_3|dx + \int _{[0,\Lambda]^2\times \re}||\nabla|^{-1}\partial_3 m_3|^2 dx
\end{equation}
subject to $m_3=\pm 1$. From this reduced functional it is clear that there is an optimal period $\omega_{domain}$ for up-down domains alternating in the $x_1$-direction and separated by Bloch walls. The wall energy of such a configuration per area in the cross section $x_1$-$x_2$ behaves as $(specific\; wall \; energy)\times(wall\; area\;per\; cross\; section)=d Q^{\frac{1}{2}} \times t \omega_{domain}^{-1}$. From the above remark on the $\dot{H}^{-\frac{1}{2}}$-norm we have by a dimensional argument that for $\omega_{domain} \ll t$, the stray-field energy scales as $\omega_{domain}$. Hence the optimal period scales as $\omega_{domain}\sim (d Q^{\frac{1}{2}} t)^{\frac{1}{2}}$, which is consistent with $\omega_{wall} \ll \omega_{domain} \ll t$ provided $dQ^{\frac{1}{2}}\ll t$. The energy per cross-sectional area scales as $(dQ^{\frac{1}{2}} t)^{\frac{1}{2}}$, beating the uniform magnetization in the above regime. However this is not the scaling of the minimal energy! Intuitively, it is advantageous to have the period $\omega_{domain}$ depending on $x_3$: for $x_3\approx \pm t$, the domain width $\omega_{domain}(x_3)$ should be vanishing in order to have mean cancellation in the "surface charges" given by ${m_3}_{\big|x_3=\pm t }$; for $x_3$ away from the surfaces, the domain width $\omega_{domain}(x_3)$ should be large in order to avoid wall energy. Such a height-dependent domains width can be realized by domain branching. However, there is no free brunch: the branching of domains means that the interfacial layers between $(0,0,1)$ and $(0,0,-1)$ are no longer vertical but tilted, hence they carry a (volume) charge $\nabla \cdot m$ and thus generate a stray field. If this tilt is small, at least in the bulk, this suggests to consider the following anisotropic reduction of \eqref{reduced-E}:
\begin{equation}\label{E-aniso}
E(m)=dQ^{\frac{1}{2}} \int_{[0,\Lambda]^2\times (-t,t)} |\nabla' m_3| dx + \int_{[0,\Lambda]^2\times \re} ||\nabla'|^{-1}\partial_3 m_3|^2 dx,
\end{equation}
subject to $m_3=\pm 1$.
Note that with our convention $m_3=0$ for $|x_3|\geq t$, the finiteness of the anisotropic version \eqref{E-aniso} of the stray field energy implies that $m_3 \stackrel{\omega*}{\rightharpoonup} 0$ for $x_3 \rightarrow \pm t$, that is, infinite branching. Now thanks to its anisotropic character, the reduction \eqref{E-aniso} has a scale invariance that allows for the following non-dimensionalization
\begin{eqnarray}\label{nondim}
x'=(d Q^{\frac{1}{2}})^{\frac{1}{3}} t^{\frac{2}{3}} \hat{x'}\;\;(\mbox{and thus}\;\; \Lambda=(d Q^{\frac{1}{2}})^{\frac{1}{3}} t^{\frac{2}{3}} \hat\Lambda),\nonumber \\
 x_3= t\hat{x_3}, \;\;\mbox{and}\;\; \frac{1}{\Lambda^2}E=(d Q^{\frac{1}{2}})^{\frac{2}{3}} t^{\frac{1}{3}}\frac{1}{\hat\Lambda^2}\hat E,
\end{eqnarray} that removes all parameters. This suggests that the domain width in the bulk scales on average as $\omega_{domain} \sim (dQ^{\frac{1}{2}})^{\frac{1}{3}} t^{\frac{2}{3}}$, and the energy per cross-sectional area as $(dQ^{\frac{1}{2}})^{\frac{2}{3}}t^{\frac{1}{3}}$, which indeed beats the unbranched case.

The passage from \eqref{E-m} to \eqref{E-aniso} has been made rigorous in \cite{OV} on the level of a $\Gamma$-convergence result for infinite cross-sectional area. A key ingredient for this is a lower bound on the minimal energy per cross-sectional area, i.e. $\frac{1}{\Lambda^2}\mbox{min}E$, that is independent of the artificial periodicity $\Lambda \gg \omega_{domain}$. It is here that the interpolation estimate from Proposition \ref{P1} comes in. Let us show how, w.l.o.g. on the level of the non-dimensionalization \eqref{nondim}. By Poincar\'e's inequality in $\hat x_3$ (recall that $m_3$ is supported in $(-1,1)$), Young's inequality and the interpolation inequality in $\hat{x'}$ (and thus for $d=2$), we have as desired:
\begin{eqnarray*}
\hat E(m)&=& \int_{[0,\hat\Lambda]^2\times (-1,1)}|\hat{\nabla}'m_3| d\hat x+ \int_{[0,\hat\Lambda]^2\times \re} | |\hat \nabla '|^{-1} \hat \partial_3 m_3|^2 d\hat x\\
&\gtrsim& \int_{-1}^1\left(\int_{[0,\hat\Lambda]^2} |\hat{\nabla}'m_3| d\hat x' +\int_{[0,\hat\Lambda]^2}||\hat \nabla|^{-1}m_3|^2 d\hat x' \right) d\hat x_3\\
&\gtrsim& \int_{-1}^1\left(\int_{[0,\hat\Lambda]^2} |\hat{\nabla}'m_3| d\hat x'\right)^{\frac{2}{3}} \left(\int_{[0,\hat\Lambda]^2}||\hat \nabla|^{-1}m_3|^2 d\hat x'\right)^{\frac{1}{3}}d\hat x_3\\
&\gtrsim& \int_{-1}^1 \int_{[0,\hat\Lambda]^2} |m_3|^{\frac{4}{3}}d\hat x' d\hat x_3 \\
&\gtrsim& \hat \Lambda^2.
\end{eqnarray*}

\textit{3. Branching in Superconductors}.

The so-called intermediate state of a type-I superconductor is characterized by penetration of the magnetic field in some parts of the material which leads to the formation of normal and superconductive domains. The superconductive regions are characterized by the expulsion of the magnetic field, this is the so called \emph{Meissner effect}. In \cite{CCKO,CKO} the mathematically rigorous study of this pattern formation problem, via energy minimization, was initiated. In particular, the authors established rigourous upper and lower bounds for a reduced energy associated to this problem, in different regimes depending on the value  of the external magnetic field (small, intermediate, or close to critical). Similarly to the case of micromagnetics, interpolation inequalities can be seen as playing a crucial role to prove an ansatz-free lower bound.
In the sequel we briefly describe the physical model.

Following \cite{COS}, for $\Lambda,t>0$, we consider a wave function $\psi: [0,\Lambda]^2\times(-t,t)\rightarrow \mathbb C$, which plays the role of the order parameter, and a vector potential $A:[0,\Lambda]^2\times(-t,t)\rightarrow \mathbb R^3$. We introduce the following quantities:
$$\mbox{the density of superconducting electrons}\quad |\psi|^2,$$
$$\mbox{the magnetic field}\quad B:=\nabla \times A,$$
$$\mbox{the covariant derivative}\quad \nabla_A \psi:=\nabla \psi - i A\psi,$$
$$\mbox{the kinetic energy} \quad |\nabla_A\psi|^2,$$
$$\mbox{and the superconductive current}\quad j:= \mathcal{I} (\overline \psi \nabla_A \psi).$$
We observe that, if $\psi$ is written in polar coordinates $\psi=\rho e^{i\theta}$, then
$$|\nabla_A \psi|^2=|\nabla \rho|^2 +\rho^2 |\nabla \theta-A|^2\quad \mbox{and}\quad j=\rho^2(\nabla \theta- A).$$
Sometimes we will use the notation $B'$ to denote the first two components of the magnetic field: $B=(B_1,B_2,B_3)=(B',B_3)$.

There are three parameters that govern the behavior of the material: the external magnetic field $B_{ext}$, the \emph{coherence length} $\xi$, which measures the typical length on which $\psi$ varies, and the \emph{penetration length} $\lambda$, which describes the typical length on which the magnetic field penetrates the superconducting region. The Ginzburg-Landau parameter is given by $\kappa=\frac{\lambda}{\xi}$.

For  any pair $(\psi,A)$ such that physically observable quantities $\rho,\; B,\;j$ are $[0,\Lambda]^2$-periodic, we define the Ginzburg-Landau functional:
\begin{eqnarray}\label{GL} 
E_0(\psi,A)&:=&\int_{[0,\Lambda]^2\times(-t,t)}\left( |\nabla_A \psi|^2 +\frac{\kappa^2}{2}(1-\rho^2)^2 \right)dx +\int_{[0,\Lambda]^2\times \re}|B-B_{ext}|^2 dx,\\
&=& \int_{[0,\Lambda]^2\times(-t,t)}\left( |\nabla \rho|^2 +\rho^2|\nabla \theta -A|^2+\frac{\kappa^2}{2}(1-\rho^2)^2\right) dx \\
&&\hspace{1em} +\int_{[0,\Lambda]^2\times \re}|B-B_{ext}|^2 dx,\nonumber
\end{eqnarray}
where $B_{ext}=(0,0,\Phi)$ is the external magnetic field and the penetration length $\lambda$ is normalized to be $1$ (that is, in this unit, $\kappa$ represents the inverse of the coherence length).
The Meissner effect, namely the fact that the kinetic energy disfavors the magnetic field where the material is superconducting (i.e. $\rho>0$) can easily be seen from this formula: In simply connected regions where $\rho^2$ is positive, the vector potential $A$ wants to be close to gradient (the gradient of the phase $\theta$) so that the magnetic field $B$ wants to be small.

The type-I superconductors correspond to the regime of $\kappa$ small ($\kappa < \sqrt 2$). In this regime, there is a positive surface tension that leads to the formation of normal and superconductive regions corresponding to $\rho \sim 0$ and $\rho \sim 1$, respectively, separated by interfaces.
Indeed, using an identity on $|\nabla_A \psi|^2$, the energy $E_0$ in \eqref{GL} can be bounded below by (see \cite{COS}, Lemma 2.3)
\begin{eqnarray}\label{E_0}
&&E_0(\psi, A)\geq \int_{[0,\Lambda]^2\times(-t,t)}\left[\left(1-\frac{\kappa}{\sqrt 2}\right)|\nabla \rho|^2 +
\left(B_3 - \frac{\kappa}{\sqrt 2} (1-\rho^2) \right)^2\right. \nonumber \\
&& \hspace{5em}\left.+  |B'|^2  -\Phi^2 +\frac{\kappa}{\sqrt 2} \Phi\right]dx + \int_{[0,\Lambda]^2\times\left(\re\setminus (-t,t)\right)}|B-B_{ext}|^2dx.
\end{eqnarray}
We observe that, under the sharp Meissner condition $\rho^2 B=0$, the sum of the first two terms on the right-hand side  can be written as
\begin{eqnarray*}
&&\int_{[0,\Lambda]^2\times(-t,t)}\left[\left(1-\frac{\kappa}{\sqrt 2}\right)|\nabla \rho|^2 +
\left(B_3 - \frac{\kappa}{\sqrt 2} (1-\rho^2) \right)^2 \right]dx\nonumber \\
&&\hspace{2em} = \int_{[0,\Lambda]^2\times(-t,t)}\left[\left(1-\frac{\kappa}{\sqrt 2}\right)|\nabla \rho|^2  + \frac{\kappa^2}{2} \chi_{\{\rho>0\}}(1-\rho^2)^2 +\chi_{\{\rho=0\}}\left(B_3 - \frac{\kappa}{\sqrt 2} \right)^2\right]dx.
\end{eqnarray*}
This is a Modica-Mortola type functional with a degenerate double-well potential given by
$$W(\rho)=\chi_{\{\rho>0\}}(1-\rho^2)^2.$$
%


After introducing the new order parameter $\chi=1-\rho^2$, and rescaling according to $x=\frac{\sqrt 2}{\kappa}\hat x$, $B=\frac{\kappa}{\sqrt 2} \hat B$, $\Phi=\frac{\kappa}{\sqrt 2} \hat \Phi$, a Modica-Mortola type argument leads to the reduced functional in the regime $\kappa\ll 1$:
\begin{equation}\label{ShIn}
\int_{[0,\Lambda]^2\times(-t,t)}\left({\frac{4}{3}|\nabla\chi|}+ |B'|^2+(B_3-{\chi})^2\right)dx
+\int_{[0,\Lambda]^2\times( \re\setminus (-t,t))}|B-(0,0,\Phi)|^2dx
\end{equation}
subject to the constraints: $\chi\in \{0,1\}$ ($\chi=0$ corresponds to the superconducting phase, and $\chi=1$ to the normal phase), $$\nabla \cdot B=0 \quad \mbox{everywhere},$$ and $$(1-\chi)B=0\quad \mbox{in}\;\;[0,\Lambda]^2\times (-t,t),$$ where $\nabla \cdot B=0$ comes from the Maxwell equations in form of $B=\nabla \times A$, while the last constraint $  (1-\chi)B=0$ represents the Meissner effect.
%

We now perform the following anisotropic rescaling: $x'=t^{\frac{2}{3}} \hat{x}'$, $x_3=t \hat{x_3}$, $B'=t^{-\frac{1}{3}} \hat{B}'$ inside the sample $[0,\Lambda]^2\times(-t,t)$, and $x=t^{\frac{2}{3}} \hat{x}$ (and thus $ \Lambda=t^{\frac{2}{3}}\hat\Lambda$), $B=t^{-\frac{1}{3}} \hat{B}$ (and thus $\Phi=t^{-\frac{1}{3}}\hat \Phi$) outside the sample.
If we define $\nu=t^{\frac{1}{3}}$, we get the energy functional in the regime $t\gg 1$ (for simplicity of notations we drop the $\hat \cdot$ on functions and variables):
\begin{equation}\label{E(chi)}
E(\chi,B)=\int_{[0,\Lambda]^2\times (-1,1)}\left(\frac{4}{3}|\nabla'\chi|+\chi |B'|^2\right)dx+\frac{1}{\nu}\int_{[0,\Lambda]^2\times (\re \setminus(-1,1))}|B-(0,0,\Phi)|^2dx,\end{equation}
subject to the constraints
\begin{equation}\label{Con-Eq}\left\{\begin{array}{lll}
\chi\in\{0,1\}, \\
\partial_3\chi+\nabla'\cdot(\chi B')=0\;\mbox{\small in sample},\\
\chi=B_3\;\mbox{\small at surface},\\
 \nabla\cdot B=0
\;\mbox{\small outside sample}.
\end{array}\right.
\end{equation}
Note that in this regime of $t\gg 1$, the penalization of $\partial_3 \chi$ fades away, while the penalization of $B_3-\chi$ turns into the hard constraint $B_3=\chi$ in the sample, so that together with the Meissner effect $B(1-\chi)=0$, $\nabla \cdot B=0$ turns into the transport equation 
$$\partial_3 \chi + \nabla'\cdot (\chi B')=0.$$

As for the case of micromagnetics, also in this case a branched pattern is observed: $\chi$ alternates between the two phases $\chi=0$ and $\chi=1$ on a length-scale which decreases while approaching the boundaries $\{x_3=\pm1\}$.

\begin{figure}[h]
	\centering

	\resizebox{10cm}{!}{
	
		\begin{tikzpicture}[scale=1]
		
			\node (myfirstpic) at (0,0) {\includegraphics[scale=0.6]{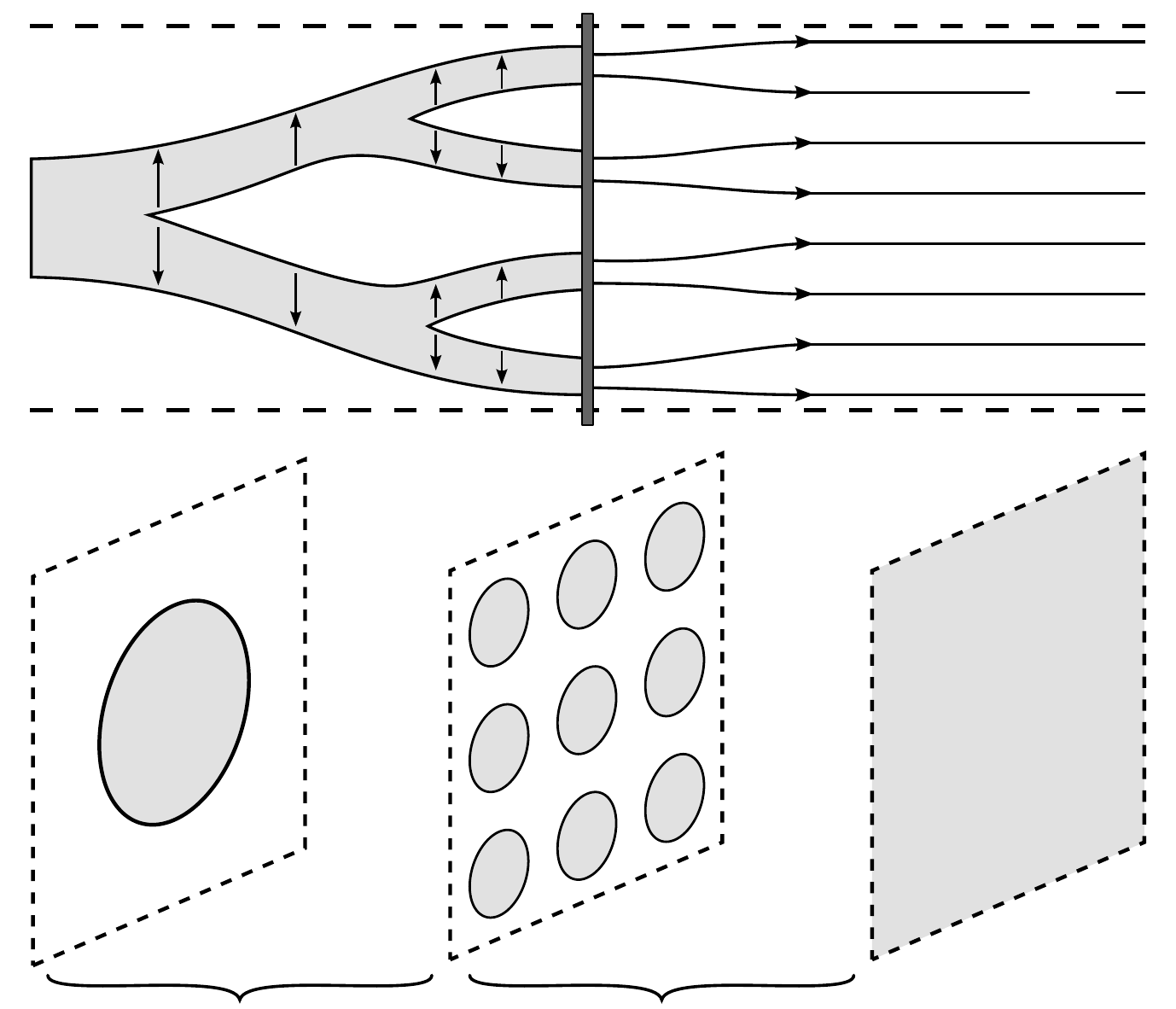}};
			
			\node at (3.48,2.95)  {\centering $B$};										
			\node at (-3.48,2.95) {\centering $B^\prime$};
			\node at (-2.5,-3.88) {\centering $W$};
			\node at (0.52,-3.88) {\centering $\dot{H}^{-\frac{1}{2}}$};			
			\node at (2.3,-0.68) {\centering $\Phi$};
			\node at (-3.7,-0.68) {\centering $\chi$};
			
		\end{tikzpicture}
		
	}
	\caption{}
\end{figure}


 Let us have now a closer look at the quantities involved in the energy $E(\chi, B)$.

 We first observe that, thanks to the last two constraints in \eqref{Con-Eq}, the last term in the energy is estimated by
 \begin{equation}\label{H}
 \int_{[0,\Lambda]^2\times (\re \setminus(-1,1))}|B-(0,0,\Phi)|^2 dx \gtrsim \| |\nabla|^{-\frac{1}{2}}(\chi_{|\{x_3=1\}}-\Phi)\|_2^2,\end{equation}
 since (cf. \eqref{H-norm})
 $$\||\nabla|^{-\frac{1}{2}}f\|_2^2=2\pi\inf \left\{\int_{[0,\Lambda]^2\times (0,\infty)} |B|^2dx \;\;:\;\;\nabla \cdot B=0,\;\;B_3(\cdot,0)=f \right\}.$$

Here, we denote by $\chi_{|_{\{x_3=1\}}}$ the \emph{weak} limit of $\chi$ as $x_3 \uparrow 1$; in particular, $\chi_{|_{\{x_3=1\}}}$ may not be a characteristic function.

 We consider now the second term $\int_{[0,\Lambda]^2\times(-1,1)}\chi |B'|^2dx$:  Using again that $\chi$ satisfies the continuity equation, cf. \eqref{Con-Eq},
$$\partial_3 \chi +\nabla'\cdot (\chi B')=0,$$
the classical Benamou-Brenier result in optimal transport theory \cite{BB} implies that for any slice $\{x_3=z\}$ for $z\in (-1,1)$
\begin{equation}\label{BB-dis}\int_{[0,\Lambda]^2\times(-1,1)}  \chi |B'|^2 dx \gtrsim W_2^2(\chi_{|_{\{x_3=z\}}}, \chi_{|_{\{x_3=1\}}}) .\end{equation}
%
%
Combining together \eqref{H} and \eqref{BB-dis}, we deduce that there exists a slice $z \in (-1,1)$ such that
\begin{equation}\label{regime3}
E(\chi,B)\gtrsim \int_{[0,\Lambda]^2}|\nabla' \chi_{|_{\{x_3=z\}}}| dx' + W_2^2(\chi_{|_{\{x_3=z\}}},\chi_{|_{\{x_3=1\}}}) + \frac{1}{\nu} \||\nabla|^{-\frac{1}{2}}(\chi_{|_{\{x_3=1\}}}-\Phi)\|_2^2.
\end{equation}
The quantities on the right-hand side are the ones involved in our Proposition \ref{cor}.

In \cite{CKO,CCKO} rigourous upper and lower bounds for the energy \eqref{E(chi)} are established. The study of minimizing configurations reveals different regimes, depending on the parameter $\nu \ll 1$ and on the value $\Phi$ of the external field.

Here we list three different regimes with the corresponding behavior of the energy:
\begin{enumerate}
\item For $\nu\ll 1$ and $1-\Phi \ll 1$ the minimal energy per area behaves like:
$$\Lambda^{-2} \min E \sim (1-\Phi)\ln^{\frac{1}{3}}(1-\Phi).$$
In this regime, the minority phase is not connected, as shown in Figure 2 below.
The interpolation estimate used for this regime is again the one with the logarithmic gain, established in Proposition \ref{P2}.
\begin{figure}[h]
	\centering	
	
	\resizebox{10cm}{!}{
	
		\begin{tikzpicture}[scale=1]
		
			\node (myfirstpic) at (0,0) {\includegraphics[scale=0.6]{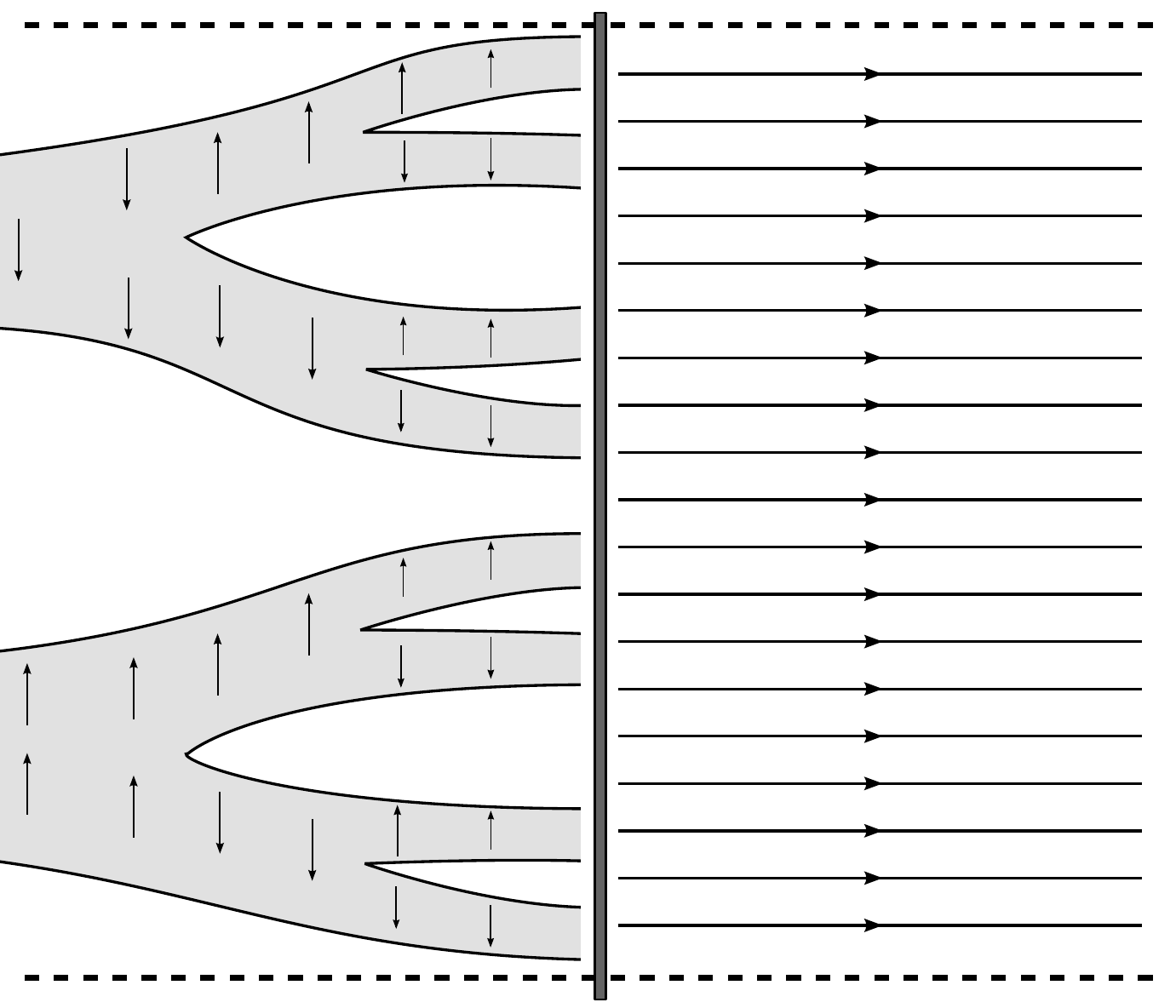}};
			
			\node at (-3.85,-.05) {$B^\prime$};			
			\node at (-3.48,1.8) {$\chi = 1$};	
			
		\end{tikzpicture}
		
	}
	\caption{}
\end{figure}

\item For $\nu\ll 1$ and $\nu^\frac{6}{7}\ll\Phi\ll 1$ the minimal energy per area behaves like:
$$\Lambda^{-2}\min E\;\sim\;{\Phi^\frac{2}{3}}.$$
In this regime, the minority phase is  connected and we have uniform branching, as shown in Figure 3 below. In this case, if we further simplify the model letting $\nu\rightarrow 0$, the last term $\||\nabla|^{-\frac{1}{2}}(\chi_{|_{\{x_3=1\}}}-\Phi)\|_2^2$ turns into the constraint $\chi_{|_{\{x_3=1\}}}\equiv\Phi$ and in order to bound from below the energy $E(\chi,B)$ it is enough to bound from below the quantity:
$$\int_{[0,\Lambda]^2}|\nabla' \chi| dx' + W_2^2(\chi,\Phi).$$
Let us show how, for this regime, our Proposition \ref{strong-W} leads to the right scaling for the minimal energy.  By \eqref{BB-dis} and Young's inequality we have for some slice $z\in (-1,1)$:
\begin{eqnarray}\label{en}
E(\chi,B)&=&\int_{[0,\Lambda]^2\times(-1,1)}(|\nabla' \chi|+\chi |B'|^2) dx \gtrsim \int_{[0,\Lambda]^2}|\nabla' \chi_{|_{\{x_3=z\}}}|dx' + W_2^2(\chi_{|_{\{x_3=z\}}},\Phi)\nonumber \\
&\gtrsim&  \left(\int_{[0,\Lambda]^2}|\nabla'\chi_{|_{\{x_3=z\}}}|dx'\right)^{\frac{2}{3}} \left(W_2^2(\chi_{|_{\{x_3=z\}}},\Phi) \right)^{\frac{1}{3}}.
\end{eqnarray}

We apply now our inequality \eqref{strong-W-eq} for $d=2$ to the function $u=\frac{\chi_{|_{\{x_3=z\}}}}{\Phi}$ in order to get
\begin{eqnarray}\label{int}
&&\left(\Phi^{-1}{\int|\nabla'\chi_{|_{\{x_3=z\}}}|}dx'\right)^\frac{2}{3}
\left(\Phi^{-1}{W_2^2(\chi_{|_{\{x_3=z\}}},\Phi)}\right)^\frac{1}{3}
\sim\hspace{1em}\|\nabla' u\|_1^\frac{2}{3}\;W_2(u,1)^\frac{2}{3}\nonumber \\
&&\hspace{1em} \gtrsim \|\max\{u\hspace{-0.3ex}-\hspace{-0.3ex}2,0\}\|_\frac{4}{3}^\frac{4}{3}
\stackrel{\chi\in\{0,1\},\Phi\ll 1}{\sim}\Lambda^2\Phi^{-\frac{1}{3}}.
\end{eqnarray}
Combining \eqref{en} and \eqref{int}, we obtain
$$E(\chi,B) \gtrsim \Lambda^2 \Phi^{\frac{2}{3}}.$$
%
%
%
%
%
%
%

\begin{figure}[h]
	\centering	
	
	\resizebox{10cm}{!}{
	
		\begin{tikzpicture}[scale=1]
		
			\node (myfirstpic) at (0,0) {\includegraphics[scale=0.6]{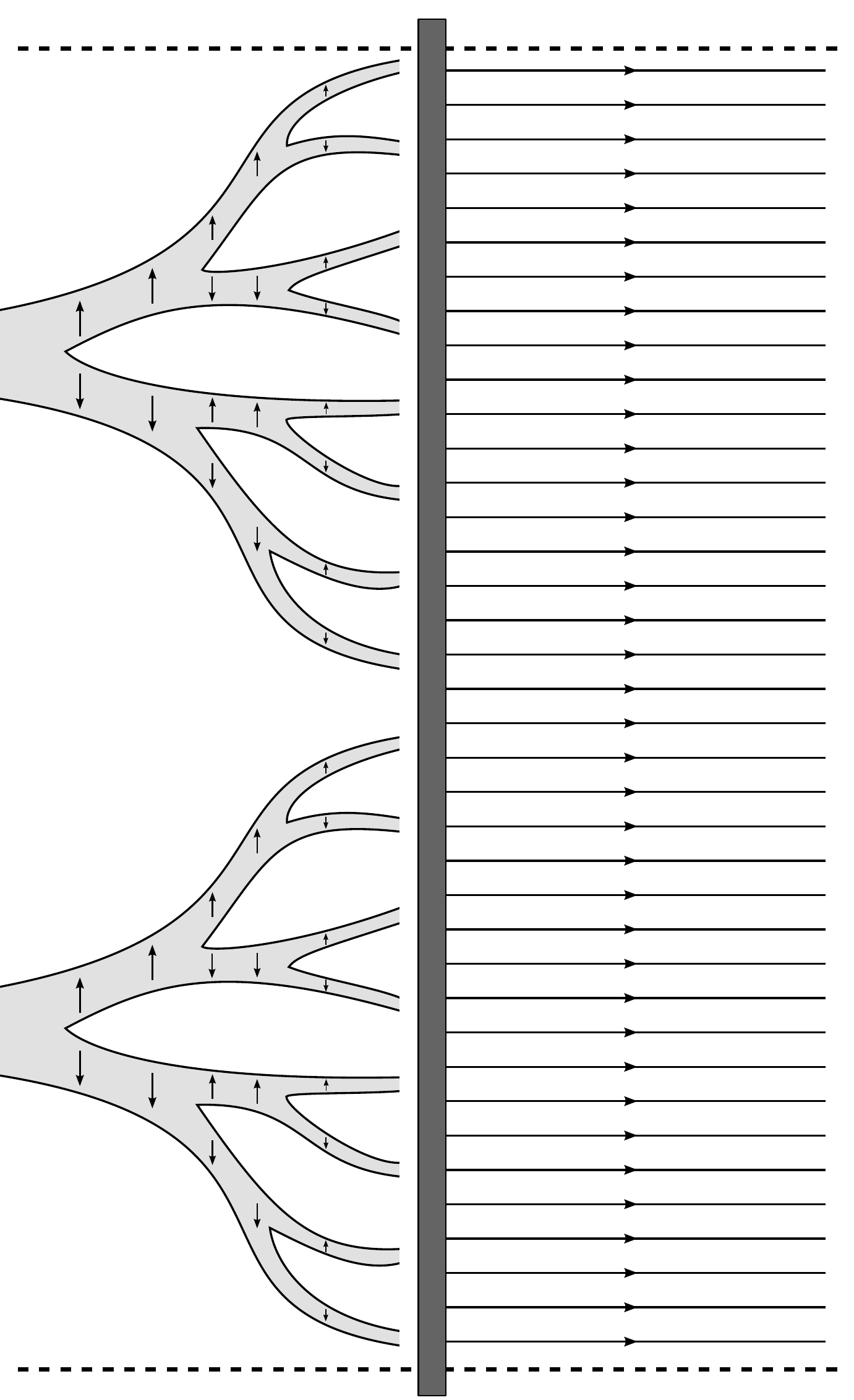}};
			
			\node at (-3.85,5.85) {$B^\prime$};		
			
		\end{tikzpicture}
		
	}
	\caption{}
\end{figure}

\item For $\nu\ll 1$ and $\Phi\ll\nu^\frac{6}{7}$ the minimal energy per area behaves like:
$$\Lambda^{-2}\min E\;\sim\;{\Phi\nu^{-\frac{2}{7}}}.$$
This regime differs from the previous regime by the non-uniformity of $\chi_{|_{\{x_3=1\}}} $ (see Figure 4 below). In this last regime, after rescaling by a suitable power of $\nu$, the lower bound for the energy can be deduced from Proposition \ref{strong-W-H} below (applied in the case of dimension $d=2$). More precisely, inequality \eqref{ineq-H} for $d=2$ reads (after taking power $\frac{3d+3}{3d+1}=\frac{9}{7}$):
\begin{equation}\label{ineq-H2}
\int_{[0,\Lambda]^2} \left( u-\nu^{\frac{7}{9}}\right)_+^{\frac{9}{7}} dx\leq  C\left(\|\nabla u\|_1+\nu^{\frac{2}{3}}W_2^2(u,v)+\nu^{-\frac{1}{3}}\| |\nabla|^{-\frac{1}{2}}(v-\Phi)\|^2_2\right).
\end{equation}
We observe that if we set $u=M \hat u$, $v=M\hat v$ (and thus $\Phi=M\hat \Phi$), $x=\ell\hat x$ (and thus $\Lambda= \ell{\hat \Lambda}$) the quantities on the right hand-side of \eqref{ineq-H2} scale as follows:
$$\begin{cases}
\|\nabla u\|_1= \ell M \|\hat \nabla \hat u \|_1,\\
W_2^2(u,v)=\ell^4 M W_2^2(\hat u, \hat v),\\
\| |\nabla|^{-\frac{1}{2}}(v-\Phi)\|^2_2=\ell^3M^2 \| |\hat\nabla|^{-\frac{1}{2}}(\hat v-\hat \Phi)\|^2_2.
\end{cases}
$$
We now choose $\ell=M=\nu^{-\frac{2}{9}}$, use the scalings above and multiply by $\nu^{\frac{4}{9}}$ in \eqref{ineq-H2} to deduce:
$$
\nu^{-\frac{2}{7}}\int_{[0,{\hat \Lambda}]^2} (\hat u -\nu)_+^{\frac{9}{7}}d \hat x \leq C\left(\|\hat \nabla \hat u \|_1 +W^2_2(\hat u, \hat v)+\nu^{-1}\| |\hat\nabla|^{-\frac{1}{2}}(\hat v-\hat \Phi)\|^2_2\right).
$$
Using \eqref{regime3} and applying the above inequality with $\hat u= \chi_{|_{\{x_3=z\}}}$ and $\hat v= \chi_{|_{\{x_3=1\}}}$,  we get the desired lower bound:
$$\Lambda^{-2}\min E\;\gtrsim\;{\Phi\nu^{-\frac{2}{7}}}.$$

\begin{figure}[h]
	\centering	
	
	\resizebox{10cm}{!}{
	
		\begin{tikzpicture}[scale=1]
		
			\node (myfirstpic) at (0,0) {\includegraphics[scale=0.6]{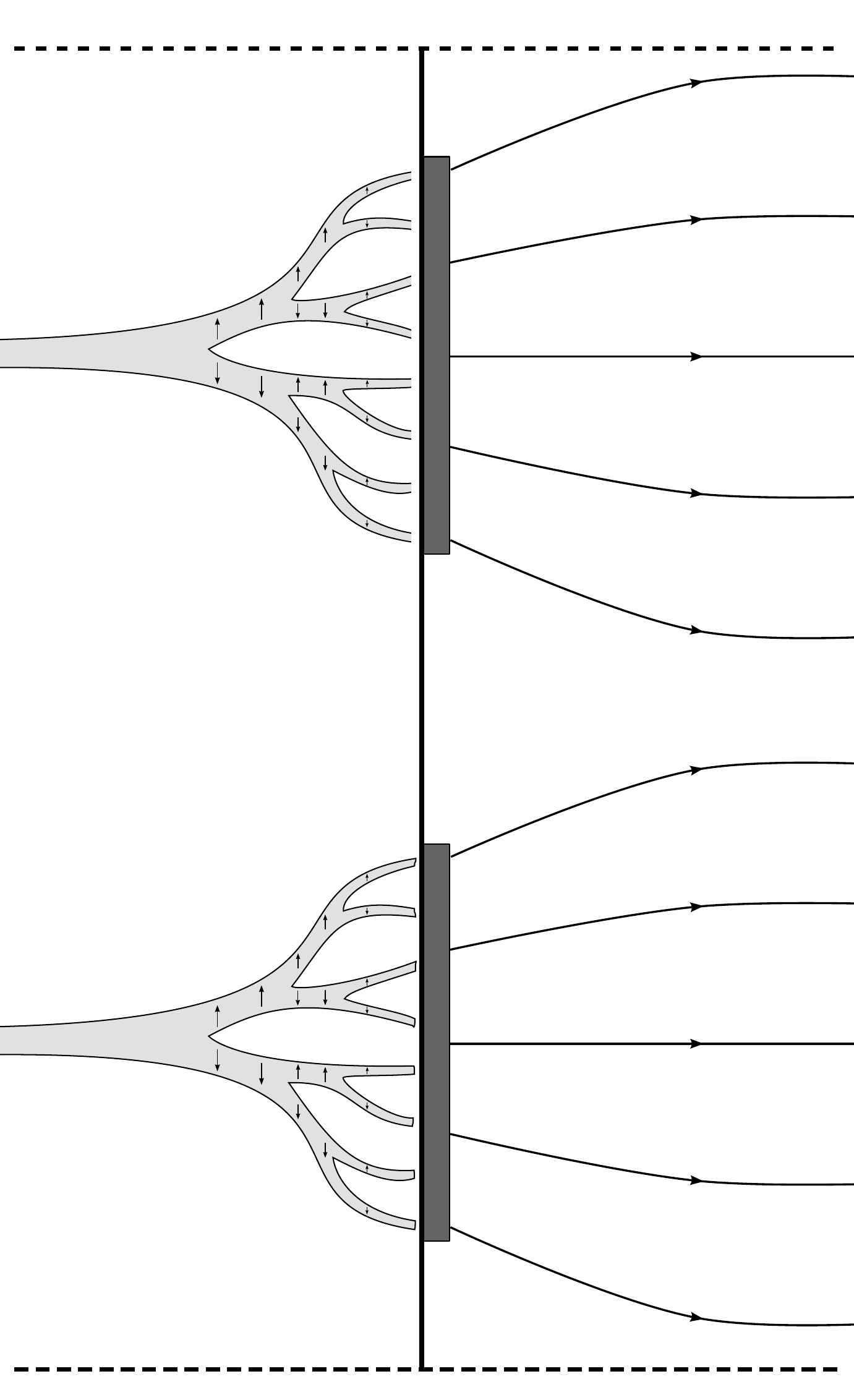}};
			
			\node at (-3.85,5.85) {$B^\prime$};		
			
		\end{tikzpicture}
		
	}
	\caption{}
\end{figure}
\end{enumerate}
\begin{prop}\label{strong-W-H}
Let $u,\:v:[0,\Lambda]^d\rightarrow \re$ be periodic functions with $u,\:v\geq 0$ and $\Lambda^{-d}\int u=\Lambda^{-d}\int v=\Phi$.

Then, there exists a constant $C>0$ only depending on $d$ such that for any $\nu>0$ with $\Phi\leq \frac{1}{C}\nu^{\frac{3d+1}{3d+3}}$ we have
\begin{equation}\label{ineq-H}
 \left|\left|\left(u-\nu^{\frac{3d+1}{3d+3}}\right)_+\right|\right|_{\frac{3d+3}{3d+1}}\leq C\left(\|\nabla u\|_1+\nu^{\frac{2}{d+1}}W_2^2(u,v)+\nu^{\frac{1-d}{d+1}}\| |\nabla|^{-\frac{1}{2}}(v-\Phi)\|^2_2\right)^{\frac{3d+1}{3d+3}}.
\end{equation}
\end{prop}
We write here the interpolation estimates \eqref{ineq-H}  in additive form, since this is the one useful for the application.  We prefer  to state Proposition \ref{cor} in multiplicative form, since this is the standard form for interpolation inequalities. In Section 5 we will see that Proposition \ref{cor} follows easily from Proposition \ref{strong-W-H}.

\section{Interpolation inequality in general dimension}
In this section we  give the proof of Proposition \ref{P1}.
We start by recalling the weak version of estimate \eqref{est1}. We remind the definition of the weak $L^p_w$-norm of a function $u$:
$$\|u\|_{w-{p}}:=\sup_{\mu>0}\mu|\{|u|\geq \mu\}|^{1/p}.$$
\begin{lem}[\cite{V}]\label{lem-weak}
 There exists a constant $C<\infty$ only depending on $d$ such that for all periodic functions $u:[0,\Lambda]^d\rightarrow \re$, with $\int u=0$,  we have
\begin{equation}\nonumber
\|u\|_{w-\frac{4}{3}}\leq
C\|\nabla u\|_1^\frac{1}{2}\;\||\nabla|^{-1}u\|_2^\frac{1}{2}.
\end{equation}
\end{lem}
This lemma is proven in \cite{V}. Here, for the sake of completeness, we give the proof of this weak estimate, since it is also useful to prove the strong version \eqref{est1}.
\begin{proof} [Proof of Lemma \ref{lem-weak}]
For simplicity of notations, in the following we will write ``$a\lesssim b$'' to mean that there exists a positive constant $C$ only depending on $d$ such that $a\leq Cb$, and $\int u$ to denote $\int_{[0,\Lambda]^d} u(x)dx$.
 By a scaling argument in $x$, it is enough to show
\begin{equation}\nonumber
\sup_{\mu\ge 0}\mu^{\frac{4}{3}}|\{|u|>\mu\}|\lesssim\|\nabla u\|_1+\||\nabla|^{-1}u\|_2^2.
\end{equation}
Indeed, the change of variables $x=L\hat x$ yields
\begin{equation}\nonumber
\sup_{\mu\ge 0}\mu^{\frac{4}{3}}|\{|u|>\mu\}|\lesssim L^{-1}\|\hat\nabla u\|_1+
L^2\||\hat\nabla|^{-1}u\|_2^2,
\end{equation}
where the symbol $\hat\nabla$ denotes the gradient with respect to the new variable $\hat x$.
The choice of $L=\|\hat\nabla u\|_1^\frac{1}{3}\||\hat\nabla|^{-1}u\|_2^{-\frac{2}{3}}$ yields
\begin{equation}\nonumber
\sup_{\mu\ge 0}\mu^{\frac{4}{3}}|\{|u|>\mu\}|\lesssim \|\hat\nabla u\|_1^{\frac{2}{3}}
\;\||\hat\nabla|^{-1}u\|_2^{\frac{2}{3}}.
\end{equation}
Raising to the power $3/4$ we get, as desired,
\begin{equation}\nonumber
\|u\|_{w-\frac{4}{3}}\lesssim \|\hat\nabla u\|_1^{\frac{1}{2}}
\;\||\hat\nabla|^{-1}u\|_2^{\frac{1}{2}}.
\end{equation}

For an arbitrary level $\mu\ge 0$ we introduce the signed characteristic function $\chi_\mu(x)$
of the $\mu$-level set of $u$:
\begin{equation}\label{chi}
\chi_\mu:=\left\{\begin{array}{ccccccc}
 1&\mbox{for}& \mu&<  &u&   &    \\
 0&\mbox{for}&-\mu&\le&u&\le& \mu\\
-1&\mbox{for}&    &   &u&<  &-\mu
\end{array}\right\}.
\end{equation}
We select now a smooth symmetric $\psi(\hat x)\ge0$ supported in $\{|\hat x|\le 1\}$ with $\int\psi d\hat x=1$ and
define the Dirac sequence $\psi_R(x)=\frac{1}{R^d}\psi(\frac{x}{R})$.
Consider then the mollification of a function $v$ on scale $R$, defined as $v_R:=\psi_R*v$.
We have the identity
\begin{equation}\nonumber
\int\chi_\mu u=\int\chi_\mu(u-u_R)+\int\chi_{\mu,R}u.
\end{equation}
Using the duality between $H^1$ and $\dot{H}^{-1}$ in the second term on the right-hand side, we get the inequality
\begin{equation}\label{a}
\mu\int|\chi_\mu|\le\int\chi_\mu u\le\|u-u_R\|_1+\|\nabla\chi_{\mu,R}\|_2\;\||\nabla|^{-1} u\|_2.
\end{equation}
On the one hand, since $\psi_R$ is supported in $\{|x|\le R\}$ we have
\begin{equation}\label{b}
\|u-u_R\|_1\le R\|\nabla u\|_1.
\end{equation}
On the other hand, using the definition of $\chi_{\mu, R}$ and of $\psi_R$, we deduce
\begin{equation}\label{c}
\|\nabla\chi_{\mu,R}\|_2\le\|\nabla\psi_R\|_1\|\chi_\mu\|_2
=R^{-1}\|\hat\nabla\psi\|_1\left(\int|\chi_{\mu}|\right)^{\frac{1}{2}}.
\end{equation}
Plugging \eqref{b} and \eqref{c} into \eqref{a}, we get
\begin{equation}\nonumber
\mu\int|\chi_\mu|\le R\|\nabla u\|_1+R^{-1}\|\hat\nabla\psi\|_1\left(\int|\chi_\mu|\right)^{\frac{1}{2}}\||\nabla|^{-1}u\|_2.
\end{equation}
The choice of $R=\mu^{-\frac{1}{3}}$ thus yields after multiplication with $\mu^{\frac{1}{3}}$:
\begin{equation}\nonumber
\mu^{\frac{4}{3}}\int|\chi_\mu|\le\|\nabla u\|_1
+\|\hat\nabla\psi\|_1\left(\mu^{\frac{4}{3}}\int|\chi_\mu|\right)^{\frac{1}{2}}\||\nabla|^{-1}u\|_2.
\end{equation}
Using Young's inequality, we may absorb the first factor
of the second term on the right-hand side
and obtain the desired estimate.

\end{proof}
We give now the proof of Proposition \ref{P1}.
The interpolation estimate \eqref{est1} was first established by Cohen, Dahmen, Daubechies, and Devore (see Theorem 1.5 in \cite{CDDD})
by wavelet methods. We give here an elementary proof, which uses an idea by Ledoux \cite{Led}.
\begin{proof}[Proof of Proposition \ref{P1}]

By scaling in $x$ as in the proof of Lemma \ref{lem-weak} it is enough to prove
\begin{equation}\nonumber
\int|u|^\frac{4}{3}\lesssim\|\nabla u\|_1+\||\nabla|^{-1}u\|_2^2.
\end{equation}
For an arbitrary level $\mu>0$ we use the signed characteristic function $\chi_\mu$ defined in \eqref{chi}. Following an idea of Ledoux \cite{Led} for the proof a similar interpolation
inequality we introduce a large parameter $M\gg1$ to be adjusted later. We consider as before $\chi_{\mu,R}=\psi_R * \chi_\mu$, where $\psi_R$ is the convolution kernel defined in the proof of Lemma \ref{lem-weak}.
We have:
\begin{eqnarray}\nonumber
\int\chi_\mu u&=&\int(\chi_\mu-\chi_{\mu,R})u+\int\chi_{\mu,R}u\nonumber\\
&=&\int_{\{|u|\le M\mu\}}(\chi_\mu-\chi_{\mu,R})u
+\int_{\{|u|>M\mu\}}(\chi_\mu-\chi_{\mu,R})u+\int\chi_{\mu,R}u.\nonumber
\end{eqnarray}
Using that $\|\chi_\mu-\chi_{\mu,R}\|_\infty\le 2$, we obtain the inequality
\begin{eqnarray}\nonumber
\int_{\{|u|>\mu\}}|u|&\le& M\mu\int|\chi_\mu-\chi_{\mu,R}|
+2\int_{\{|u|>M\mu\}}|u|+\int\chi_{\mu,R}u\nonumber\\
&\le& M\mu R\int|\nabla\chi_\mu|
+2\int_{\{|u|>M\mu\}}|u|+\int\chi_{\mu,R}u.\nonumber
\end{eqnarray}
We multiply with $\mu^{-\frac{2}{3}}$ and choose $R=\mu^{-\frac{1}{3}}$ as in the proof of Lemma \ref{lem-weak}.
Integrating over $\mu\in(0,\infty)$ and using the duality between $\dot{H}^{-1}$ and $H^1$, we get
\begin{eqnarray}\nonumber
\lefteqn{\int_0^\infty\mu^{-\frac{2}{3}}\int_{\{|u|>\mu\}}|u|dxd\mu}\nonumber\\
&\le& M\int_0^\infty\int|\nabla\chi_\mu|dxd\mu
+2\int_0^\infty\mu^{-\frac{2}{3}}\int_{\{|u|>M\mu\}}|u|dxd\mu\nonumber\\
&&+\int\left(\int_0^\infty\mu^{-\frac{2}{3}}\chi_{\mu,R}d\mu\right)\:u\:dx\nonumber\\
&\le& M\int_0^\infty\int|\nabla\chi_\mu|dxd\mu
+2\int_0^\infty\mu^{-\frac{2}{3}}\int_{\{|u|>M\mu\}}|u|dxd\mu\nonumber\\
&&+\left\|\nabla\left(\int_0^\infty\mu^{-\frac{2}{3}}\chi_{\mu,R}d\mu\right)\right\|_2
\;\||\nabla|^{-1}u\|_2\nonumber,
\end{eqnarray}
where we keep the abbreviation $R=\mu^{-\frac{1}{3}}$.

On the left-hand side we have
\begin{equation}\nonumber
\int_0^\infty\mu^{-\frac{2}{3}}\int_{\{|u(x)|>\mu\}}|u(x)|dxd\mu
=\int|u(x)|\int_0^{|u(x)|}\mu^{-\frac{2}{3}}d\mu dx
=3\int|u|^\frac{4}{3}.
\end{equation}
We address the three terms on the right-hand side one by one. We start by the second one:
\begin{eqnarray}\nonumber
\lefteqn{\int_0^\infty\mu^{-\frac{2}{3}}\int_{\{|u(x)|>M\mu\}}|u(x)|dxd\mu}\nonumber\\
&=&\int|u(x)|\int_0^{M^{-1}|u(x)|}\mu^{-\frac{2}{3}}d\mu dx
=3M^{-\frac{1}{3}}\int|u|^\frac{4}{3}.\nonumber
\end{eqnarray}
We now address the first term.
By the coarea formula we get
\begin{equation}\nonumber
\int_0^\infty\int|\nabla\chi_\mu|dxd\mu=\int_0^\infty({\rm Per}(\{u>\mu\})+{\rm Per}(\{u<-\mu\}))d\mu
=\|\nabla u\|_1,
\end{equation}
where ${\rm Per}(A)$ denotes the perimeter of $A$.
Finally we consider the last term (with $R':={\mu'}^{-\frac{1}{3}}$):
\begin{eqnarray}\nonumber
\lefteqn{\left\|\nabla\left(\int_0^\infty\mu^{-\frac{2}{3}}\chi_{\mu,R}d\mu\right)\right\|_2^2}\nonumber\\
&=&\int_0^\infty\int_0^\infty\mu^{-\frac{2}{3}}{\mu'}^{-\frac{2}{3}}
\int\nabla\chi_{\mu,R}\cdot\nabla\chi_{\mu',R'}dx d\mu' d\mu\nonumber\\
&=&2\int_0^\infty\int_0^\mu \mu^{-\frac{2}{3}}{\mu'}^{-\frac{2}{3}}
\int(-\Delta)\chi_{\mu,R}\;\chi_{\mu',R'}dx d\mu' d\mu\nonumber\\
&=&2\int_0^\infty\int_0^\mu \mu^{-\frac{2}{3}}{\mu'}^{-\frac{2}{3}}
\int\psi_{R'}*(-\Delta\psi_{R})*\chi_{\mu}\;\chi_{\mu'}dx d\mu' d\mu\nonumber\\
&\le&2\int_0^\infty\int_0^\mu \mu^{-\frac{2}{3}}{\mu'}^{-\frac{2}{3}}
\|\psi_{R'}\|_1\;\|\Delta\psi_{R}\|_1\;\|\chi_{\mu}\|_1\;\|\chi_{\mu'}\|_\infty d\mu' d\mu\nonumber\\
&=&2\|\hat\Delta\psi\|_1\int_0^\infty\int_0^\mu \mu^{-\frac{2}{3}}{\mu'}^{-\frac{2}{3}}
R^{-2}\;\|\chi_{\mu}\|_1 d\mu' d\mu\nonumber\\
&=&2\|\hat\Delta\psi\|_1\int_0^\infty\int_0^\mu{\mu'}^{-\frac{2}{3}}d\mu'
\;\|\chi_{\mu}\|_1 d\mu\nonumber\\
&=&6\|\hat\Delta\psi\|_1\int_0^\infty\mu^{\frac{1}{3}}
\;|\{|u|>\mu\}| d\mu\;=\;
6\|\hat\Delta\psi\|_1\int\int_0^{|u(x)|}\mu^{\frac{1}{3}}d\mu dx\nonumber\\
&=&\frac{9}{2}\|\hat\Delta\psi\|_1\int|u|^{\frac{4}{3}}.\nonumber
\end{eqnarray}
These inequalities combine to
\begin{eqnarray}\nonumber
\lefteqn{3\int|u|^{\frac{4}{3}}}\nonumber\\
&\le& M\|\nabla u\|_{1}
+6M^{-\frac{1}{3}}\int|u|^{\frac{4}{3}}
+\left(\frac{9}{2}\|\hat\Delta\psi\|_1\int|u|^{\frac{4}{3}}\right)^\frac{1}{2}
\;\||\nabla|^{-1}u\|_2\nonumber.
\end{eqnarray}
We obtain the desired estimate by absorbing the middle right-hand side term for $M$ large enough and absorbing the first factor of the last right-hand side term by Young's inequality.
\end{proof}
\section{Proof of Proposition \ref{P2}}
In this section we prove Proposition \ref{P2}. We begin by recalling a geometric version of estimate \eqref{inter-est},
which was established by Conti, Niethammer, and one of the authors in \cite{CNO}.
\begin{lem}[\cite{CNO}]\label{P2-geom}
Let $\chi:[0,\Lambda]^2\rightarrow\{0,1\}$ be a periodic characteristic function with volume fraction
$\Phi:=\Lambda^{-2}\int\chi\ll1$.

Then, there exists a constant $C>0$ such that
\begin{equation}\label{geom-est}
\Phi\ln^{\frac{1}{3}}\frac{1}{\Phi}\leq C\left(\Lambda^{-2}\int|\nabla\chi|\right)^{\frac{2}{3}}
\left(\Lambda^{-2}\int||\nabla|^{-1}(\chi-\Phi)|^2\right)^{\frac{1}{3}}.
\end{equation}
\end{lem}
The proof of this Lemma made use of the following geometric construction, that plays a crucial role also in the proof of our strong estimate \eqref{inter-est}.
\begin{lem}[\cite{CNO}]\label{geom}
 For any periodic function $\chi:[0,\Lambda]^2\rightarrow\{0,1\}$ and $R\ll L$ there
exists a potential $\phi_{R,L}(x)\in[0,1]$ such that
\begin{eqnarray}
\int\chi&\lesssim& R\int|\nabla\chi|+\int\chi\phi_{R,L},\label{1}\\
\int\max\{-\Delta\phi_{R,L},0\}&\lesssim&R^{-2}\left(\ln^{-1}\frac{L}{R}\right)\int\chi,\label{2}\\
\int\phi_{R,L}&\lesssim&L^2R^{-2}\int\chi.\label{3}
\end{eqnarray}
\end{lem}
We note that for $L=R$ we could just choose $\phi_{R,L}=\psi_R*\chi=\chi_R$;
the interest here is the logarithmic gain $\ln^{-1}\frac{L}{R}$ for $L\gg R$.
\begin{oss}
We observe, for later reference, that for any function $\phi'(x)\in[0,1]$ we have
\begin{equation}\label{2a}
\int\nabla\phi_{R,L}\cdot\nabla\phi'\lesssim R^{-2}\left(\ln^{-1}\frac{L}{R}\right)\int\chi.
\end{equation}
Indeed, we have
\begin{eqnarray*}\nonumber
\lefteqn{\int\nabla\phi_{R,L}\cdot\nabla\phi'}\nonumber\\
&=&\int(-\Delta\phi_{R,L})\phi'
{\le}\int\max\{-\Delta\phi_{R,L},0\}\phi'\nonumber\\
&{\le}&\int\max\{-\Delta\phi_{R,L},0\}\le R^{-2}\left(\ln^{-1}\frac{L}{R}\right)\int\chi ,
\end{eqnarray*}
where in the first two inequalities we have used $\phi'\geq 0$ and $\phi'\leq 1$ respectively. The last inequality follows by applying \eqref{2}.

In particular, we obtain for $\phi'=\phi_{R,L}$
\begin{equation}\label{4}
\int|\nabla\phi_{R,L}|^2\lesssim R^{-2}\left(\ln^{-1}\frac{L}{R}\right)\int\chi.
\end{equation}
\end{oss}
This type of geometric construction was first used by Choksi, Conti, Kohn, and  one of the authors in \cite{CCKO} in the context of branched patterns in superconductors.

For the convenience of the reader we reproduce a version of the proof of Lemma \ref{geom}.
\begin{proof}[Proof of Lemma \ref{geom}]
We split the proof in two steps.

{\bf Step 1.} In the first step, we construct a set $\Omega_R$ that covers most of $\{\chi=1\}$ (see {\it Claim 1} below)
and has radius of curvature $\lesssim R$ (see {\it Claim 2} below). As before, let $\chi_R=\psi_R*\chi$ denote the mollification
of $\chi$ on scale $R$. We define
\begin{equation}\nonumber
\Omega_R:=\{\chi_R>1/2\}.
\end{equation}

This time, we take the non-smooth ``Dirac sequence''
\begin{equation}\nonumber
\psi_R(x)
=\left\{
\begin{array}{ccc}
\frac{4}{\pi R^2}&\mbox{for}&|x|<  \frac{R}{2}\\
0                &\mbox{for}&|x|\ge\frac{R}{2}
\end{array}\right\},
\end{equation}
so that $\Omega_R$ can be characterized via the density of $\{\chi=1\}$ in balls of
radius $R/2$ as follows
\begin{equation}\nonumber
\Omega_R=\left\{x\;\big|\;|\{\chi=1\}\cap B_\frac{R}{2}(x)|>{\textstyle\frac{1}{2}}|B_\frac{R}{2}(x)|\right\}.
\end{equation}
We show now the two following claims.

{{\it Claim 1}}: for $\Omega_R$ defined above we have
\begin{eqnarray}\label{cl-1}
\int\chi&\lesssim& R\int|\nabla\chi|+\int_{\Omega_R}\chi.
\end{eqnarray}
%
Indeed,
$$\int\chi-\int_{\Omega_R}\chi=|\{\chi=1\}\cap\{\chi_R\le 1/2\}|\le2\|\chi-\chi_R\|_1\le2R\int|\nabla\chi|.$$

{\it Claim 2}:
There exists a finite number $N$ of points $y_i\in \Omega_R$ for $i=1,...,N$, such that
$$\Omega_R\subset\bigcup_{i=1}^N B_R(y_i)\quad\mbox{and}\quad
N\lesssim\frac{1}{R^2}\int\chi.$$

Indeed, let $\{y_1,...,y_N\}$ be maximal with the property that $B_\frac{R}{2}(y_i)\cap B_\frac{R}{2}(y_j)=\emptyset$
for any $i,j=1,...,N,$ with $i\neq j$. The first part of the claim follows from the maximality of $\{y_1,...,y_N\}$; indeed
if there were an $y_0\in\Omega_R$ with $y_0\not\in B_R(y_i)$
and thus $B_\frac{R}{2}(y_0)\cap B_\frac{R}{2}(y_i)=\emptyset$ for all $i=1,...,N$, also the strictly larger
set $\{y_0,y_1,...,y_N\}$ would be admissible.

The second part of the claim can be seen as follows:
\begin{equation}\nonumber
N\frac{\pi}{4} R^2=\sum_{i=1}^N|B_\frac{R}{2}(y_i)|
<2\sum_{i=1}^N|\{\chi=1\}\cap B_\frac{R}{2}(y_i)|\le2|\{\chi=1\}|,
\end{equation}
where in the first inequality, we have used that for $y_i\in \Omega_R$  we have
$|\{\chi=1\}\cap B_\frac{R}{2}(y)|>\frac{1}{2}|B_\frac{R}{2}(y)|$. In the last inequality we have used
the pairwise disjointness of $\{B_\frac{R}{2}(y_i)\}_{i=1,...,N}$.

{\bf Step 2.}
In the second step, we construct the potential $\phi_{R,L}$. We introduce the
capacity potential $\hat\phi_{R,L}$ of $B_R(0)$ in $B_L(0)$ given by

\begin{equation}\nonumber
\hat\phi_{R,L}(\hat x):=\left\{\begin{array}{ccccccc}
1&\mbox{for}&&&|\hat x|&\le&R\\
\frac{\ln\frac{L}{|\hat x|}}{\ln\frac{L}{R}}&\mbox{for}&R&\le&|\hat x|&\le&L\\
0&\mbox{for}&L&\le&|\hat x|&&
\end{array}\right\}\;\in\;[0,1].
\end{equation}

We define
\begin{equation}\nonumber
\phi_{R,L}(x):=\max_{i=1,...,N}\hat\phi_{R,L}(x-y_i)\;\in\;[0,1].
\end{equation}

{{\it Claim 3}}: we have $$\int\chi\lesssim R\int|\nabla\chi|+\int\chi\phi_{R,L}.$$
Indeed in {\it Claim 2} we have seen that $\Omega_R \subset \bigcup_{i=1}^N B_R(y_i)$, therefore, since by definition $\hat \phi_{R,L}=1$ in $B_R(0)$, we deduce that $\phi_{R,L}=1$ on $\Omega_R$. This, together with
{\it Claim 1} implies {\it Claim 3} and thus \eqref{1}.

{{\it Claim 4}}: We have $$\int\phi_{R,L}\lesssim L^2R^{-2}\int\chi.$$ Indeed, by the definition of $\phi_{R,L}$, we have
\begin{equation}\nonumber
\int\phi_{R,L}\le N\int\hat\phi_{R,L}\lesssim N L^2
\lesssim L^2 R^{-2} \int\chi,
\end{equation}
where we have used the upper bound on $N$ established in {\it Claim 2}.

{{\it Claim 5}}: $$\int\max\{-\Delta\phi_{R,L},0\}\lesssim R^{-2}(\ln^{-1}\frac{L}{R})\int\chi.$$
Indeed, using the well-known fact that the singular part of
$(-\Delta)\max\{\phi_1,\phi_2\}$ is negative, we conclude similarly to
the previous step:
\begin{eqnarray}\nonumber
\lefteqn{\int\max\{-\Delta\phi_{R,L},0\}\le N\int\max\{-\Delta\hat\phi_{R,L},0\}}\nonumber\\
&=& N 2\pi\ln^{-1}\frac{L}{R}
\lesssim R^{-2}\left(\ln^{-1}\frac{L}{R}\right) \int\chi.\nonumber
\end{eqnarray}
This concludes the proof of Lemma \ref{geom}.
\end{proof}
For the convenience of the reader, we reproduce the proof of Lemma \ref{P2-geom} from \cite{CNO}.
\begin{proof}[Proof of Lemma \ref{P2-geom}]
 By the three properties of the geometric construction we have
\begin{eqnarray*}
\int\chi
&\lesssim&
R\int|\nabla\chi|+\int\phi_{R,L}\chi\\
&=&R\int|\nabla\chi|+\int\phi_{R,L}(\chi-\Phi)+\Phi\int\phi_{R,L}\\
&\le&R\int|\nabla\chi|+\left(\int|\nabla\phi_{R,L}|^2\int||\nabla|^{-1}(\chi-\Phi)|^2\right)^{\frac{1}{2}}
+\Phi\int\phi_{R,L}\\
&\lesssim&R\int|\nabla\chi|+\left(R^{-2}\left(\ln^{-1}\frac{L}{R}\right)\int\chi
\int||\nabla|^{-1}(\chi-\Phi)|^2\right)^{\frac{1}{2}}
+\Phi\left(\frac{L}{R}\right)^2\int\chi.
\end{eqnarray*}
We first absorb the factor $\int \chi$ of the middle right-hand side term by Young's inequality to get
\begin{eqnarray*}
\int\chi&\lesssim&
R\int|\nabla\chi|+R^{-2}\left(\ln^{-1}\frac{L}{R}\right)
\int||\nabla|^{-1}(\chi-\Phi)|^2
+\Phi \left(\frac{L}{R}\right)^{2}\int\chi.
\end{eqnarray*}
In order to absorb the last right-hand side term, we choose $L$ to be a small but order one multiple of $\Phi^{-\frac{1}{2}}R$.
Since $L$ is a small multiple of $\Phi^{-\frac{1}{2}}R$, we have $\Phi(\frac{L}{R})^2\ll 1$ so that indeed we can absorb;
since it is an order one multiple of $\Phi^{-\frac{1}{2}}R$ and $\Phi\ll 1$, we have $L\gg R$ and $\ln\frac{L}{R}\sim\ln\frac{1}{\Phi}$.
Hence we obtain:
\begin{eqnarray*}
\int\chi&\lesssim&
R\int|\nabla\chi|+R^{-2}\left(\ln^{-1}\frac{1}{\Phi}\right)
\int||\nabla|^{-1}(\chi-\Phi)|^2.
\end{eqnarray*}
We finally optimize in $R$ by choosing
$R=(\int|\nabla\chi|)^{-\frac{1}{3}}((\ln^{-1}\frac{1}{\Phi})\int||\nabla|^{-1}(\chi-\Phi)|^2)^{\frac{1}{3}}$, and we get
\begin{eqnarray*}
\int\chi&\lesssim&
\left(\ln^{-\frac{1}{3}}\frac{1}{\Phi}\right)\left(\int|\nabla\chi|\right)^{\frac{2}{3}}
\left(\int||\nabla|^{-1}(\chi-\Phi)|^2\right)^{\frac{1}{3}}.
\end{eqnarray*}
Dividing by $\Lambda^2$ and multiplying by $\ln^{\frac{1}{3}}\frac{1}{\Phi}$, we obtain the desired
estimate.
\end{proof}
As in the previous section,  we recall here the weak version of our interpolation inequality \eqref{inter-est} in dimension $2$, which was proven in the  PhD thesis of Viehmann \cite{V}.
\begin{prop}[\cite{V}]\label{weak-est2}
Let $u:[0,\Lambda]^2\rightarrow \re$ be a periodic function with $u\geq -1$ and $\int u=0$.

Then, there exists a constant $C>0$ such that
\begin{equation}\nonumber
\sup_{\mu\ge e}\mu(\ln^{\frac{1}{4}}\mu)\,|\{|u|>\mu\}|^\frac{3}{4}\lesssim
\|\nabla u\|_1^\frac{1}{2}\;\||\nabla|^{-1}u\|_{2}^\frac{1}{2}.
\end{equation}\end{prop}

\begin{proof}[Proof of Proposition \ref{weak-est2}] By Lemma \ref{lem-weak} and
by a scaling argument in $x$, it is enough to show
\begin{equation}\nonumber
\sup_{\mu\gg 1}\mu^{\frac{4}{3}}(\ln^{\frac{1}{3}}\mu)\,|\{|u|>\mu\}|\lesssim
\|\nabla u\|_1+\||\nabla|^{-1}u\|_{2}^2.
\end{equation}
For a given level $\mu\gg 1$ we consider the characteristic function $\chi_\mu(x)\in\{0,1\}$ of the corresponding
level set of $u$, that is
\begin{equation}\nonumber
\{\chi_\mu=1\}=\{u>\mu\}.
\end{equation}
For given length scales $R\ll L$ (to be chosen later) let $\phi_{\mu,R,L}$ be the potential
constructed in Lemma \ref{geom} based on $\chi_\mu$. According to Lemma \ref{geom} we have
\begin{eqnarray*}
\int\chi_\mu&\lesssim&R\int|\nabla\chi_\mu|+\int\chi_\mu\phi_{\mu,R,L}.
\end{eqnarray*}
 Using that $\phi_{\mu,R,L}\geq 0$ and the crucial assumption $u\ge-1$, we rewrite this as
\begin{eqnarray*}
\lefteqn{R\int|\nabla\chi_\mu|+\int\chi_\mu\phi_{\mu,R,L}}\\
&{\le}&R\int|\nabla\chi_\mu|+\mu^{-1}\int\chi_\mu\phi_{\mu,R,L} u\\
&=&R\int|\nabla\chi_\mu|+\mu^{-1}\int\chi_\mu\phi_{\mu,R,L}(u+1)-\mu^{-1}\int\chi_\mu\phi_{\mu,R,L}\\
&{\le}&R\int|\nabla\chi_\mu|+\mu^{-1}\int\phi_{\mu,R,L}(u+1)-\mu^{-1}\int\chi_\mu\phi_{\mu,R,L}\\
&=&R\int|\nabla\chi_\mu|+\mu^{-1}\int\phi_{\mu,R,L}u+\mu^{-1}\int(1-\chi_\mu)\phi_{\mu,R,L}\\
&{\le}&R\int|\nabla\chi_\mu|+\mu^{-1}\int\phi_{\mu,R,L}u+\mu^{-1}\int\phi_{\mu,R,L}\\
&\le&R\int|\nabla\chi_\mu|+\mu^{-1}\left(\int|\nabla\phi_{\mu,R,L}|^2\int||\nabla|^{-1}u|^2\right)^\frac{1}{2}
+\mu^{-1}\int\phi_{\mu,R,L}.
\end{eqnarray*}
We now insert estimates \eqref{3} and \eqref{4} from Lemma \ref{geom} to obtain
\begin{eqnarray*}
\int\chi_\mu
&\lesssim&
R\int|\nabla\chi_\mu|\\
&&+\mu^{-1}\left(R^{-2}\left(\ln^{-1}\frac{L}{R}\right)\int\chi_\mu\int||\nabla|^{-1}u|^2\right)^\frac{1}{2}
+\mu^{-1}\left(\frac{L}{R}\right)^2\int\chi_\mu.
\end{eqnarray*}
In order to absorb the last right-hand side term, we choose $L$ to be a small but order one multiple of $\mu^\frac{1}{2}R$.
Since $L$ is a small multiple of $\mu^\frac{1}{2}R$, we have $\mu^{-1}(\frac{L}{R})^2\ll 1$ so that indeed we can absorb;
since it is an order one multiple of $\mu^\frac{1}{2}R$ and $\mu\gg 1$, we have $L\gg R$ and $\ln\frac{L}{R}\sim\ln\mu$.
Hence we obtain:
\begin{eqnarray*}
\int\chi_\mu
&\lesssim&
R\int|\nabla\chi_\mu|+\mu^{-1}\left(R^{-2}(\ln^{-1}\mu)\int\chi_\mu\int||\nabla|^{-1}u|^2\right)^\frac{1}{2}.
\end{eqnarray*}
In order to absorb the factor $\int\chi_\mu$ of the last remaining right-hand side term, we use Young's inequality and we get
\begin{equation}\label{weak}
{|\{u>\mu\}|=\int\chi_\mu}
\lesssim
R\int|\nabla\chi_\mu|+\mu^{-2}R^{-2}(\ln^{-1}\mu)\int||\nabla|^{-1}u|^2.
\end{equation}
By the coarea formula, we have $\int_{\frac{\mu}{2}}^\mu\int|\nabla\chi_\mu|dx d\mu\le\int|\nabla u|$
so that there exists a $\mu'\in[\frac{\mu}{2},\mu]$ with $\mu\int|\nabla\chi_{\mu'}|\le 2\int|\nabla u|$.
Using \eqref{weak} for $\mu$ replaced by $\mu'$ we thus have
\begin{eqnarray*}
|\{u>\mu'\}|
&\lesssim&
R\mu^{-1}\int|\nabla u|+{\mu'}^{-2}R^{-2}(\ln^{-1}{\mu'})\int||\nabla|^{-1}u|^2,
\end{eqnarray*}
which because of $\mu'\in[\frac{\mu}{2},\mu]$ turns into
\begin{eqnarray*}
|\{u>\mu\}|
&\lesssim&
R\mu^{-1}\int|\nabla u|+\mu^{-2}R^{-2}(\ln^{-1}\mu)\int||\nabla|^{-1}u|^2.
\end{eqnarray*}
We multiply with $\mu^{\frac{4}{3}}\ln^\frac{1}{3}\mu$ and we get
\begin{eqnarray*}
\lefteqn{\mu^{\frac{4}{3}}\ln^{\frac{1}{3}}\mu|\{u>\mu\}|}\nonumber\\
&\lesssim&
R(\mu\ln\mu)^{\frac{1}{3}}\int|\nabla u|+R^{-2}(\mu\ln\mu)^{-\frac{2}{3}}\int||\nabla|^{-1}u|^2.
\end{eqnarray*}
The choice of $R=(\mu\ln\mu)^{-\frac{1}{3}}$ yields the desired estimate.
\end{proof}
We now give the proof of our strong interpolation inequality in dimension $2$.
\begin{proof}[Proof of Proposition \ref{P2}]
By an approximation argument, we can assume that $u$ is a step function (indeed all the quantities appearing in inequality \eqref{inter-est} well behave under approximation by step functions).
 By a scaling argument in $x$ and the result in Proposition \ref{P1}, it is enough to show for $M\gg 1$:
\begin{equation}\nonumber
\int_{\{u\ge 2M\}}u^{\frac{4}{3}}\ln^{\frac{1}{3}}u\lesssim
\|\nabla u\|_1+\||\nabla|^{-1}u\|_{2}^2.
\end{equation}
We consider an arbitrary level $\mu\ge M\gg 1$ and start as in the proof of Proposition \ref{weak-est2}, considering the potential $\phi_{\mu,R,L}$. Observe that, since we are assuming that $u$ is a step function, then $\phi_{\mu,R,L}$ is piecewise constant as a function of $\mu$, and therefore it is measurable in $\mu$. This will be important later since we will integrate $\phi_{\mu,R,L}$ in $d\mu$.
For $L$ chosen such that $(\frac{L}{R})^2\sim\mu$, we get
\begin{eqnarray*}
\int\chi_\mu&\lesssim&R\int|\nabla\chi_\mu|+\mu^{-1}\int\phi_{\mu,R,L} u.
\end{eqnarray*}
But we now rather proceed as in Proposition \ref{P1}. We multiply with $(\mu\ln\mu)^\frac{1}{3}$,
choose $R=(\mu\ln\mu)^{-\frac{1}{3}}$ and integrate in $\mu\in(M,\infty)$ for $M\gg 1$:
\begin{eqnarray}\label{P2a}
\lefteqn{\int_M^\infty(\mu\ln\mu)^\frac{1}{3}\int\chi_\mu dx\,d\mu}\nonumber\\
&\lesssim&\int_M^\infty\int|\nabla\chi_\mu| dx d\mu
+\int\left(\int_M^\infty\frac{\ln^\frac{1}{3}\mu}{\mu^\frac{2}{3}}\phi_{\mu,R,L}d\mu\right) \:u\: dx\nonumber \\
&\le&\|\nabla u\|_1
+\left\|\nabla\left(\int_M^\infty\frac{\ln^\frac{1}{3}\mu}{\mu^\frac{2}{3}}\phi_{\mu,R,L}d\mu\right)\right\|_2\;\||\nabla|^{-1} u\|_2.
\end{eqnarray}
On the last right-hand side term we argue along the lines of Proposition \ref{P1}, now using the property \eqref{2a}
of our geometric construction,  that is
\begin{equation}\label{P2b}\int\nabla\phi_{\mu,R,L}\cdot\nabla\phi_{\mu',R',L'}dx\lesssim
\frac{1}{R^2}\frac{1}{\ln\frac{L}{R}}\int\chi_\mu dx,\end{equation}
where $R'$ and $L'$ are related to $\mu'$ like $R$ and $L$ to $\mu$, that is, $R':=({\mu'}\ln \mu')^{-\frac{1}{3}}$, $(\frac{L'}{R'})^2\sim\mu'$.
Using \eqref{P2b} and by the choice of $R$ and $L$, we get
\begin{eqnarray}\nonumber
\lefteqn{\left\|\nabla\left(\int_M^\infty\frac{\ln^\frac{1}{3}\mu }{\mu^\frac{2}{3}}\phi_{\mu,R,L}d\mu\right)\right\|_2^2}\nonumber\\
&=&\int_M^\infty\int_M^\infty\frac{\ln^\frac{1}{3}\mu }{   \mu^\frac{2}{3}}
                             \frac{\ln^\frac{1}{3}\mu'}{{\mu'}^\frac{2}{3}}
\int\nabla\phi_{\mu,R,L}\cdot\nabla\phi_{\mu',R',L'}dx d\mu' d\mu\nonumber\\
&=&2\int_M^\infty\int_M^\mu \frac{\ln^\frac{1}{3}\mu }{   \mu^\frac{2}{3}}
                            \frac{\ln^\frac{1}{3}\mu'}{{\mu'}^\frac{2}{3}}
\int\nabla\phi_{\mu,R,L}\cdot\nabla\phi_{\mu',R',L'}dx d\mu' d\mu\nonumber\\
&\stackrel{\eqref{P2b}}\lesssim&\int_M^\infty\int_M^\mu \frac{\ln^\frac{1}{3}\mu }{   \mu^\frac{2}{3}}
                                  \frac{\ln^\frac{1}{3}\mu'}{{\mu'}^\frac{2}{3}}
\frac{1}{R^2}\frac{1}{\ln\frac{L}{R}}\int\chi_\mu dx d\mu' d\mu\nonumber\\
&{\sim}&
\int_M^\infty\int_M^\mu\frac{\ln^\frac{1}{3}\mu'}{{\mu'}^\frac{2}{3}}d\mu'
\;\int\chi_\mu dx d\mu\nonumber\\
&\lesssim&
\int_M^\infty(\mu\ln\mu)^\frac{1}{3}
\;\int\chi_\mu dx d\mu\nonumber.
\end{eqnarray}
Hence, coming back to \eqref{P2a}, we can absorb this term by Young's inequality and obtain
\begin{equation}\nonumber
\int_M^\infty(\mu\ln\mu)^\frac{1}{3}
\;\int\chi_\mu dx d\mu\lesssim\|\nabla u\|_1+\||\nabla|^{-1}u\|_2^2.
\end{equation}
We conclude by observing that for $M\gg1$
\begin{equation}\nonumber
\int_M^\infty(\mu\ln\mu)^\frac{1}{3}\int\chi_\mu dx\,d\mu
=\int_{\{u>M\}}\int_M^{u(x)}(\mu\ln\mu)^\frac{1}{3}d\mu\,dx
{\gtrsim}\int_{\{u>2M\}}u^\frac{4}{3}\ln^{\frac{1}{3}}u.
\end{equation}
\end{proof}

\section{Proof of Propositions \ref{strong-W} and \ref{strong-W-H}}
In this section we prove Propositions \ref{strong-W},\ref{strong-W-H}, and \ref{cor}. The two main ingredients in the proofs are the geometric construction of Lemma \ref{geom} and
the Kantorovich duality for the Wasserstein distance.

\begin{oss}\label{geom-d}
In the proof of Proposition \ref{strong-W} we need the analog of the geometric construction of Lemma \ref{geom} in any dimension $d$.
Following the proof of Lemma \ref{geom} it is easy to see that given a function $\chi: [0,\Lambda]^d\rightarrow \{0,1\}$ there exists a set $\Omega_R$, which is defined as
$$\Omega_R=\left\{x\;\big|\;|\{\chi=1\}\cap B_\frac{R}{2}(x)|>{\textstyle\frac{1}{2}}|B_\frac{R}{2}(x)|\right\},$$
and $N$ points $y_i\in \Omega_R$ for $i=1,...,N$ such that
\begin{equation}\label{gw1}
 \Omega_R \subset \bigcup_{i=1}^N B_R(y_i)\quad \mbox{and}\quad N\lesssim \frac{1}{R^d}\int \chi,
\end{equation}
where $\{y_1,...,y_N\}$ is maximal with the property that $B_{R/2}(y_i)\cap B_{R/2}(y_j)=\emptyset$ for every $i,j=1,...,N$ such that $i\neq j.$

Similarly to Lemma \ref{geom}, we want to define now a potential $\phi_R$ associated to $\chi$. Since here we are not interested in the logarithmic behaviour of the potential, it is enough to define $\phi_R$ as the characteristic function of $\bigcup_{i=1}^N B_R(y_i)$.
With this choice of $\phi_R$, similarly to \eqref{1}, we have
\begin{equation}\label{gw3}
 \int \chi\lesssim R\int |\nabla \chi| + \int \chi \phi_R.
\end{equation}
\end{oss}
We are now ready to give the proof of our Proposition \ref{strong-W}.

\begin{proof}[Proof of Proposition \ref{strong-W}]
As before, by approximation we may assume that $u$ is a step function.
By a scaling argument in $x$ it is enough to show that there exists a constant $C$ only depending on $d$ such that
\begin{equation}\label{W}
 \|(u-C)_+\|_{\frac{2+3d}{3d}}^{\frac{2+3d}{3d}}\leq C \left(\|\nabla u\|_1+ W_2^2(u,1)\right).
\end{equation}
%
%
To make the proof more readable, we divide it in three steps.

\textbf{Step 1.} We start as in the proof of Proposition \ref{weak-est2} using our geometric construction.
For a given level $\mu$, let as before $\chi_\mu(x)\in\{0,1\}$ denote the characteristic function of the set $\{ u>\mu\}$ and
 let $\phi_{\mu,R}$ be the potential associated to $\chi_\mu$ defined in Remark \ref{geom-d}. By \eqref{gw3} we have
\begin{eqnarray}\label{W-1}
 \int \chi_\mu &\lesssim& R\int |\nabla \chi_\mu|+\int \chi_\mu \phi_{\mu,R}\nonumber \\
&\leq& R\int |\nabla \chi_\mu|+\frac{1}{\mu}\int \phi_{\mu,R} u.
\end{eqnarray}
Let $\varepsilon$ be a small parameter to be adjusted later.
We multiply \eqref{W-1} by $\mu^{(2+3d)/(3d)}$, we choose $ R=C_1^{1/2}\mu^{-2/(3d)}$ (where $C_1$ is a dimensional constant to be specified later) and we integrate in $\int \frac{d\mu}{\mu}$ for $\mu \geq 1/\varepsilon^d$, to get
\begin{equation}\label{W1}
\begin{split}
\int_{1/\varepsilon^d}^{+\infty} \mu^{\frac{2+3d}{3d}}\int\chi_\mu dx \frac{d\mu}{\mu}&\lesssim \int_{1/\varepsilon^d}^{+\infty}\int|\nabla \chi_\mu|dx d\mu  \\
&\hspace{0.5em} + \int\left(\int_{1/\varepsilon^d}^{+\infty}\mu^{\frac{2}{3d}}\phi_{\mu,R}(x)\frac{d\mu}{\mu}\right) u(x)dx.\end{split}\end{equation}
Using the coarea formula as before, the first term on the right-hand side is estimated as follows
$$ \int_{1/\varepsilon^d}^{+\infty}\int|\nabla \chi_\mu|dx d\mu\leq \|\nabla  u\|_1.$$
To estimate the second term on the right-hand side of \eqref{W1}, we set
\begin{equation}\label{phi}\varphi(x):=\int_{1/\varepsilon^d}^{+\infty}\mu^{\frac{2}{3d}}\phi_{\mu,R}(x)\frac{d\mu}{\mu}.\end{equation}
%
%
 Using the Kantorovich duality \eqref{Kan} with $v\equiv 1$, with $\psi$ replaced by $-\psi$, and with cost function $c(x)=|x-y|^2/\varepsilon^2$, we have that
\begin{equation}\label{Ka}
 \int \varphi(x) u(x) dx \leq \frac{1}{\varepsilon^2}W_2^2( u,1)+\int \psi(y)dy,
\end{equation}
where
\begin{equation}\label{psi}\psi(y):=\sup_x\left\{\varphi(x)-\frac{|x-y|^2}{\varepsilon^2}\right\}=\sup_x\left\{\int_{1/\varepsilon^d}^{+\infty}\mu^{\frac{2}{3d}}\phi_{\mu,R}(x)\frac{d\mu}{\mu}-\frac{|x-y|^2}{\varepsilon^2}\right\}.\end{equation}
Combining all together in \eqref{W1} we get
\begin{equation}\label{W2}\int_{1/\varepsilon^d}^{\infty} \mu^{\frac{2+3d}{3d}}\int \chi_\mu dx \frac{d\mu}{\mu}\lesssim \|\nabla u\|_1 + \frac{1}{\varepsilon^2}W_2^2(u,1) + \int \psi(y)dy,\end{equation} where $\psi$ is defined in \eqref{psi}.

\textbf{Step 2.} In this step, we estimate the term $\int \psi$. We will show that, for $\varepsilon$ small enough, it can be absorbed on the left-hand side and this will conclude the proof. In this step, we will assume that the following inequality holds:
\begin{equation}\label{claim1}
\begin{split}
\psi(y)&\leq \sup_x\left\{\int_{1/\varepsilon^d}^{+\infty}\mu^{\frac{2}{3d}}\phi_{\mu,R}(x)\frac{d\mu}{\mu}-\frac{|x-y|^2}{\varepsilon^2}\right\}_+ \\
&\leq 2\int_{1/\varepsilon^d}^{+\infty} \sup_x \left\{C_1\mu^{\frac{2}{3d}}\phi_{\mu,R}(x)-\frac{|x-y|^2}{\varepsilon^2}\right\}_+\frac{d\mu}{\mu}=:\int_{1/\varepsilon^d}^{+\infty}\widetilde \psi_\mu(y) \frac{d\mu}{\mu},
\end{split}
\end{equation}
where, for simplicity of notations, we write $\sup\{f\}_+$ in place of $\left(\sup\{f\}\right)_+$.
%
We will prove this inequality in Step 3 below.
%
%
By \eqref{claim1} we have that
\begin{equation}\label{int-psi}
\int \psi(y)dy \leq \int\left(\int_{1/\varepsilon^d}^{+\infty} \widetilde \psi_\mu(y)\frac{d\mu}{\mu} \right) dy =\int_{1/\varepsilon^d}^{+\infty}\left(\int  \widetilde \psi_\mu(y)dy\right)\frac{d\mu}{\mu}.\end{equation}
We recall that (see Remark \ref{geom-d}) for any $\mu$, $\phi_{\mu,R}$ is the characteristic function of the union of $N$
balls $B_R(y_i)$ for $i=1,...,N$, where $N$ is bounded above by $\frac{1}{R^d}\int \chi_\mu$.

This implies that
\begin{eqnarray*}
 \widetilde \psi_\mu(y)&=&2\sup_x\left\{C_1\mu^{\frac{2}{3d}}\phi_{\mu,R}(x)-\frac{|x-y|^2}{\varepsilon^2}\right\}_+\\
&=&2\left\{\begin{array}{ll} C_1\mu^{\frac{2}{3d}}   \quad \mbox{if}\; y\in B_R(y_i)\;\;\mbox{for some}\; i=1,...,N \\
\max\left\{C_1\mu^{\frac{2}{3d}}-d_y^2/\varepsilon^2, 0 \right\} \quad \mbox{if}\; y\notin B_R(y_i)\;\;\mbox{for every}\; i,\end{array}\right.
\end{eqnarray*}
where $d_y:=\textit{dist}\left(y, \bigcup_{i=1}^N B_R(y_i)\right)$, with $\textit{dist}(y, A)$ denoting the distance between $y$ and the set $A$. Observe that if $y$ is such that
$d_y\geq C_1^{1/2}\varepsilon\mu^{\frac{1}{3d}}$ then $\widetilde \psi_\mu(y)=0$.
Moreover by the choice $R=C_1^{1/2}\mu^{-2/(3d)}$ and by our assumption $\mu\geq 1/\varepsilon^d$, we have $C_1^{1/2}\varepsilon\mu^{\frac{1}{3d}}\geq R$.
Thus $\widetilde \psi_\mu$ is supported in the union of $N$ balls $B_l(y_i)$, with radius $l=R+C_1^{1/2}\varepsilon \mu^{\frac{1}{3d}}\leq 2C_1^{1/2}\varepsilon \mu^{\frac{1}{3d}}.$

Hence we have
\begin{equation*}
\int \widetilde\psi_\mu(y)dy
\lesssim N \mu^{\frac{2}{3d}}l^d \stackrel{\eqref{gw1}}\lesssim \varepsilon^d\mu^{\frac{2+d}{3d}}\cdot \frac{1}{R^d}\int \chi_\mu dx
\lesssim \varepsilon^d\mu^{\frac{2+3d}{3d}}\int\chi_\mu dx.
\end{equation*}
Using \eqref{int-psi} this yields
$$\int \psi(y) dy \leq \int_{1/\varepsilon^d}^\infty \left(\int \widetilde \psi_\mu(y)dy\right) \frac{d\mu}{\mu} \lesssim \varepsilon^d\int_{1/\varepsilon^d}^\infty \mu^{\frac{2+3d}{3d}}\int\chi_\mu dx \frac{d\mu}{\mu}.$$

Combining all together in \eqref{W2} we get that there exists a constant $\widetilde C >0$ only depending on $d$ such that
$$\int_{1/\varepsilon^d}^\infty \mu^{\frac{2+3d}{3d}}\int \chi_\mu dx \frac{d\mu}{\mu}\leq \widetilde C \left(\|\nabla u\|_1 + \frac{1}{\varepsilon^2}W_2^2(u,1) +
 \varepsilon^d\int_{1/\varepsilon^d}^\infty \mu^{\frac{2+3d}{3d}}\int\chi_\mu dx \frac{d\mu}{\mu}\right).$$
Choosing $\varepsilon=(2\widetilde C)^{-1/d}$ we can absorb the last term on the right-hand side to get
$$\int_{2\widetilde C}^{\infty} \mu^{\frac{2+3d}{3d}}\int \chi_\mu dx \frac{d\mu}{\mu}\leq 2\widetilde C\left(\|\nabla u\|_1+(2 \widetilde C)^{2/d}W_2^2(u,1)\right).$$
Evaluating the integral in $\mu$ on the left-hand side we deduce \eqref{W} with $C=4\widetilde C^2$.

\textbf{Step 3.}
In this last step we show \eqref{claim1} and therefore we conclude the proof. By dilation in $\mu$, it is enough to prove that there exists a dimensional constant $C_1$ such that
\begin{equation}\label{claim1a}
\sup_x\left\{\int_{1}^{+\infty}\mu^{\frac{2}{3d}}\phi_{\mu,R}(x)\frac{d\mu}{\mu}-\frac{|x-y|^2}{\varepsilon^2}\right\}_+
\leq 2\int_{1}^{+\infty} \sup_x \left\{C_1\mu^{\frac{2}{3d}}\phi_{\mu,R}(x)-\frac{|x-y|^2}{\varepsilon^2}\right\}_+\frac{d\mu}{\mu}
\end{equation}
We start by proving two technical estimates.\\
\textsl{Claim A.} Given a function $f$, we have the following relations between (discrete) dyadic sums and (continuum) logarithmic integrals:
\begin{equation}\label{claimA}
\frac{1}{2} \int_1^{2} \sum_{k=0}^\infty f(\theta 2^k)d\theta \leq \int_1^{+\infty} f(\mu)\frac{d\mu}{\mu}\leq \int_1^{2} \sum_{k=0}^\infty f(\theta 2^k)d\theta.
\end{equation}
The proof of Claim A is trivial; indeed, using the change of variable $\mu=\theta 2^k$, we have
$$\int_1^2 \sum_{k=0}^\infty f(\theta 2^k)d\theta=\sum_{k=0}^\infty \int_1^{2} f(\theta 2^k)d\theta=\sum_{k=0}^\infty \int_{2^k}^{2^{k+1}}f(\mu)\frac {d\mu}{2^k}\begin{cases}&\displaystyle\leq 2\int_1^{+\infty} f(\mu)\frac{d\mu}{\mu}\\
&\displaystyle\geq \int_1^{+\infty} f(\mu)\frac{d\mu}{\mu}\end{cases}.$$

\textsl{Claim B.} For any fixed $\theta\in \re$ and $k\in \N$, let $\phi_{\theta,k}$ be a characteristic function. Then, for $p>0$, we have that for any $x$ the following estimate holds:
\begin{equation}\label{claimB}
\sum_{k=0}^\infty (\theta2^k)^p \phi_{\theta,k}(x)\leq \frac{2^p}{2^p-1} \sup_{0\leq k <\infty} \{(\theta 2^k)^p\phi_{\theta,k}(x)\}.
\end{equation}
To prove this claim, we set
\begin{equation}\label{K}
K(x):=\sup \{k| \:\phi_{\theta,k}(x)\neq 0\}.
\end{equation}

 %
%
If $K(x)=\infty$ inequality \eqref{claimB} holds trivially since the right-hand side is infinite. If $K(x)< \infty$ then we have
\begin{eqnarray*}
\sum_{k=0}^\infty (\theta 2^k)^p\phi_{\theta,k}(x)&\leq& \sum_{k=0}^{K(x)} (\theta 2^k)^p = \theta^p\frac{2^{p(K(x)+1)}-1}{2^p-1}\\
&\leq& \left(\theta 2^{K(x)}\right)^p\frac{2^p}{2^p-1} = \frac{2^p}{2^p-1} \sup_{0\leq k<\infty} \{(\theta 2^k)^p\phi_{\theta,k}(x)\},\end{eqnarray*}
which concludes the proof of Claim B.

%
%
Now, using \eqref{claimA} and \eqref{claimB} we deduce that \eqref{claim1a} holds.
Indeed, recalling that $\phi_{\mu,R}$ is the characteristic function of a finite union of balls (depending on $\mu$), we have
\begin{eqnarray*}
&&\sup_x \left\{\int_{1}^\infty\mu^{\frac{2}{3d}} \phi_{\mu,R}(x)\frac{d\mu}{\mu}-\frac{|x-y|^2}{\varepsilon^2}\right\}_+\\
&&\hspace{1em}\stackrel{\eqref{claimA}}\leq  \sup_x \left\{\int_1^2\left(\sum_{k=0}^\infty (\theta 2^k)^{\frac{2}{3d}} \phi_{\theta 2^k,R}(x)\right) d\theta -\frac{|x-y|^2}{\varepsilon^2}\right\}_+\\
&&\hspace{1em}=\sup_x\left\{\int_1^2 \left(\sum_{k=0}^\infty (\theta 2^k)^{\frac{2}{3d}} \phi_{\theta 2^k,R}(x)-\frac{|x-y|^2}{\varepsilon^2}\right)d\theta\right\}_+\\
&&\hspace{1em}\leq \sup_x\left\{\int_1^2 \left(\sum_{k=0}^\infty (\theta 2^k)^{\frac{2}{3d}} \phi_{\theta 2^k,R}(x)-\frac{|x-y|^2}{\varepsilon^2}\right)_+d\theta\right\}\\
&&\hspace{1em}=\int_1^2 \sup_x \left(\sum_{k=0}^\infty (\theta 2^k)^{\frac{2}{3d}} \phi_{\theta 2^k,R}(x)-\frac{|x-y|^2}{\varepsilon^2}\right)_+d\theta\\
&&\hspace{1em}\stackrel{\eqref{claimB}}\leq \int_1^2 \sup_x \left(C_1\sup_{0\leq k<\infty}\{(\theta 2^k)^{\frac{2}{3d}}\phi_{\theta 2^k,R}(x)\}-\frac{|x-y|^2}{\varepsilon^2}\right)_+ d\theta \\
&&\hspace{1em} = \int_1^2 \sup_x \left(\sup_{0 \leq k<\infty}\left\{C_1(\theta 2^k)^{\frac{2}{3d}}\phi_{\theta 2^k,R}(x)-\frac{|x-y|^2}{\varepsilon^2}\right\}\right)_+ d\theta \\
&&\hspace{1em} = \int_1^2 \sup_{0\leq k<\infty} \left(\sup_x\left\{C_1(\theta 2^k)^{\frac{2}{3d}}\phi_{\theta 2^k,R}(x)-\frac{|x-y|^2}{\varepsilon^2}\right\}\right)_+ d\theta \\
&&\hspace{1em}\leq  \int_1^2 \sum_{k=0}^\infty \left(\sup_x\left\{C_1(\theta 2^k)^{\frac{2}{3d}}\phi_{\theta 2^k,R}(x)-\frac{|x-y|^2}{\varepsilon^2}\right\}\right)_+ d\theta\\
&&\hspace{1em}\stackrel{\eqref{claimA}} \leq 2
\int_1^{+\infty} \sup_x \left\{C_1\mu^{\frac{2}{3d}}\phi_{\mu,R}(x)-\frac{|x-y|^2}{\varepsilon^2}\right\}_+\frac{d\mu}{\mu},
\end{eqnarray*}
where $C_1=\frac{2^{\frac{2}{3d}}}{2^{\frac{2}{3d}}-1}$ (cf. Claim B).

This concludes the proof of \eqref{claim1} and thus of the proposition.

\end{proof}

We give now the proof of the interpolation estimate in additive form stated in Proposition \ref{strong-W-H}.
\begin{proof}[Proof of Proposition \ref{strong-W-H}]
We follow the same strategy of the proof of Proposition \ref{strong-W} and we divide again the proof in three steps.

\textbf{Step 1.}
We start by showing that inequality \eqref{ineq-H} holds with $\nu=1$, that is we want to prove that for every $u,v>0$ such that $\Lambda^{-d}\int  u=\Lambda^{-d}\int v=\Phi$, there exists a constant $C$ with $\Phi\leq C^{-1}$, such that
\begin{equation}\label{W-H}
\int\left(u-1\right)_+^{\frac{3d+3}{3d+1}}\leq C\left( \|\nabla u\|_1+W_2^2(u,v)+\||\nabla|^{-\frac{1}{2}}(v-\Phi)\|_2^2\right).\end{equation}

We begin as in the proof of Proposition \ref{strong-W}, assuming that $u$ is a step function and using our geometric construction. We have (cf. \eqref{W-1})
\begin{equation}\label{W-H-1}\int \chi_{\mu}\lesssim R\int |\nabla \chi_\mu| +\frac{1}{\mu}\int \phi_{\mu,R} u.\end{equation}
We multiply \eqref{W-H-1} by $\mu^{\frac{3d+3}{3d+1}}$, choose $R=C_2^{1/2}\mu^{-\frac{2}{3d+1}}$ (where $C_2$ is a dimensional constant to be specified later) and  integrate in $\int \frac{d\mu}{\mu}$ for $\mu \geq 1$, to get
\begin{eqnarray}\label{W-H-2}
&&\int_1^\infty \mu^{\frac{3d+3}{3d+1}}\int \chi_\mu(x)dx \frac{d\mu}{\mu}\nonumber \\
&&\hspace{2em}\lesssim \int_1^\infty \int |\nabla \chi_\mu(x)|dx d\mu +\int\left(\int_1^\infty \mu^{\frac{2}{3d+1}} \phi_{\mu,R}(x)\frac{d\mu}{\mu}\right)u(x) dx.\end{eqnarray}
Using the coarea formula as before, the first term on the right-hand side is estimated as
$$\int_1^\infty \int |\nabla \chi_\mu(x)|dx d\mu \leq \|\nabla u\|_1.$$
To estimate the second term on the right-hand side we proceed as in the proof of Proposition \ref{strong-W}. By analogy with \eqref{phi}, we set
$$\varphi(x):=\int_1^\infty \mu^{\frac{2}{3d+1}}\phi_{\mu,R}(x)\frac{d\mu}{\mu}.$$
Using again the Kantorovich duality, we obtain
\begin{equation}\label{K-H}
\int \varphi(x)u(x)dx \leq W_2^2(u,v)+\int \psi(y)v(y)dy,\end{equation}
where analogously to \eqref{claim1} we have
\begin{eqnarray*}
\psi(y)&=&\sup_x\{\varphi(x)-|x-y|^2\}\\
&\leq&\sup_x\left\{\int_1^{+\infty}\mu^{\frac{2}{3d+1}}\phi_{\mu,R}(x)\frac{d\mu}{\mu}-|x-y|^2\right\}_+\\
&\leq& 2\int_1^{+\infty} \sup_x \left\{C_2\mu^{\frac{2}{3d+1}}\phi_{\mu,R}(x)-|x-y|^2\right\}_+\frac{d\mu}{\mu},\end{eqnarray*}
where $C_2=\frac{2^{\frac{2}{3d+1}}}{2^{\frac{2}{3d+1}}-1}$.

%
Plugging \eqref{K-H} into \eqref{W-H-2}, we get
\begin{equation}\label{W-H-3}\int_1^\infty \mu^{\frac{3d+3}{3d+1}}\int \chi_\mu(x)dx \frac{d\mu}{\mu}\lesssim \|\nabla u\|_1 + W_2^2(u,v) + \int \psi(y)v(y)dy.\end{equation}

\textbf{Step 2.} In this step we estimate the term $\int \psi v$.
In order to do that, we define the functions
$$\widetilde\psi_\mu(y):=2\sup_x\left\{C_2\mu^{\frac{2}{3d+1}} \phi_{\mu,R}(x)-|x-y|^2\right\}_+\quad\mbox{and}\quad \widetilde \psi(y):=\int_1^{\infty}\widetilde\psi_\mu(y)\frac{d\mu}{\mu}.$$
The second term on  the right-hand side of \eqref{K-H} is bounded by
\begin{equation}\label{dual}
\begin{split}
\int \psi v &\leq \int \widetilde \psi v= \int \widetilde \psi (v-\Phi) + \Phi \int \widetilde \psi \\
& \leq \|\ |\nabla|^{\frac{1}{2}}\widetilde \psi\|_2\| |\nabla|^{-\frac{1}{2}}(v-\Phi)\|_2+\Phi \int \widetilde \psi .\end{split}\end{equation}
We give now an estimate for the quantity
$\| |\nabla|^{\frac{1}{2}} \widetilde \psi\|_2$.
We recall that here $\dot H^{1/2}$ refers to the homogeneous fractional Sobolev space, endowed with the seminorm defined by
$$\| |\nabla|^{\frac{1}{2}}f\|_2(\re^d):=\int_{\re^d}\int_{\re^d}\frac{|f(x)-f(\overline x)|^2}{|x-\overline x|^{d+1}}dx d\overline x.$$
We have
\begin{eqnarray*}
\||\nabla|^{\frac{1}{2}}\widetilde \psi\|^2_2&=& \left\| |\nabla|^{\frac{1}{2}}\int_1^{\infty}\widetilde \psi_{\mu}\frac{d\mu}{\mu}\right\|^2_2\\
&=&\int_{\re^d}\int_{\re^d}\left(\int_1^{\infty}\int_1^\infty
\frac{(\widetilde \psi_{\mu}(y)-\widetilde \psi_{\mu}(\overline y))(\widetilde \psi_{\mu'}(y)-\widetilde \psi_{\mu'}(\overline y))}{|y-\overline y|^{d+1}}\frac{d\mu}{\mu}\frac{d\mu'}{\mu'}\right)dy d\overline y\\
&=& 2\int_1^{\infty}\frac{d\mu}{\mu}\int_1^\mu \frac{d\mu'}{\mu'} \left(\int_{\re^d}\int_{\re^d}
\frac{(\widetilde \psi_{\mu}(y)-\widetilde \psi_{\mu}(\overline y))(\widetilde \psi_{\mu'}(y)-\widetilde \psi_{\mu'}(\overline y))}{|y-\overline y|^{d+1}}dy d\overline y\right)\\
&=:&2\int_1^{\infty}\frac{d\mu}{\mu}\int_1^\mu \frac{d\mu'}{\mu'}I^{\mu,\mu'},
\end{eqnarray*}
where in the last equality we have changed the order of integration and we have used the symmetry between $\mu$ and $\mu'$.
To estimate the quantity $I^{\mu,\mu'}$ we proceed as follows. First recall that, since $\phi_{\mu,R}$ is the characteristic function of the union of $N$ balls $B_R(y_i)$, we have
\begin{eqnarray}\label{psi-mu-H}
 \widetilde \psi_\mu(y)&=&2\left(\sup_x\left\{C_2\mu^{\frac{2}{3d+1}}\phi_{\mu,R}(x)-|x-y|^2\right\}\right)_+\nonumber \\
&=&2\left\{\begin{array}{ll} C_2\mu^{\frac{2}{3d+1}}   \quad \mbox{if}\; y\in B_R(y_i)\;\;\mbox{for some}\; i \\
\max\left\{C_2\mu^{\frac{2}{3d+1}}-d_y^2, 0 \right\} \quad \mbox{if}\; y\notin B_R(y_i)\;\;\mbox{for every}\; i.\end{array}\right.
\end{eqnarray}
where $d_y:=\textit{dist}(y, \bigcup_{i=1}^N B_R(y_i))\}$. Observe that if $y$ is such that
$d_y\geq C_2^{1/2}\mu^{\frac{1}{3d+1}}$ then $\widetilde \psi_\mu(y)=0$, and that by the choice $R=C_2^{1/2}\mu^{-2/(3d+1)}$ we have $\mu^{\frac{1}{3d+1}}\geq R$, for $\mu\geq 1$. Thus $\widetilde \psi_\mu$ is supported in the union of $N$ balls $B_l(y_i)$, with radius $l= R+ C_2^{1/2}\mu^{\frac{1}{3d+1}}\leq 2 C_2^{1/2}\mu^{\frac{1}{3d+1}}$.
Moreover recall that the number $N$ of balls is bounded by
\begin{equation}\label{N1}
N \lesssim \frac{1}{R^d}\int \chi_\mu.\end{equation}
By \eqref{psi-mu-H}, we see that $\widetilde \psi_\mu$ satisfies
\begin{equation}\label{L-infty} |\widetilde \psi_\mu|\lesssim \mu^{\frac{2}{3d+1}},\end{equation}
\begin{equation}\label{grad}
|\nabla \widetilde \psi_\mu|\lesssim \frac{\mu^{\frac{2}{3d+1}}}{l} \lesssim \mu^{\frac{1}{3d+1}}.\end{equation}
Using the fact that $\widetilde \psi_\mu$ is supported in $\cup_{i=1}^N B_l(y_i)$ and the symmetry between $y$ and $\overline y$, we have
\begin{eqnarray*}
I^{\mu,\mu'}&\leq& 2\sum_{i=1}^{N}\int_{B_l(y_i)}\int_{\re^n}\frac{(\widetilde \psi_{\mu}(y)-\widetilde \psi_{\mu}(\overline y))(\widetilde \psi_{\mu'}(y)-\widetilde \psi_{\mu'}(\overline y))}{|y-\overline y|^{d+1}}dy d\overline y\\
&=& 2\sum_{i=1}^{N}\int_{B_l(y_i)}\int_{\{\overline y:|y-\overline y|\leq r\}}\frac{(\widetilde \psi_{\mu}(y)-\widetilde \psi_{\mu}(\overline y))(\widetilde \psi_{\mu'}(y)-\widetilde \psi_{\mu'}(\overline y))}{|y-\overline y|^{d+1}}dy d\overline y\\
&&\hspace{0.2em} + 2\sum_{i=1}^{N}\int_{B_l(y_i)}\int_{\{\overline y:|y-\overline y|> r\}}\frac{(\widetilde \psi_{\mu}(y)-\widetilde \psi_{\mu}(\overline y))(\widetilde \psi_{\mu'}(y)-\widetilde \psi_{\mu'}(\overline y))}{|y-\overline y|^{d+1}}dy d\overline y\\
&=& 2\sum_{i=1}^N\left(I^{\mu, \mu'}_{i,1}+I^{\mu, \mu'}_{i,2}\right).\end{eqnarray*}
To bound the first term $I^{\mu, \mu'}_{i,1}$, we use that $l\lesssim\mu^{1/(3d+1)}$, the gradient bound \eqref{grad} for $\widetilde \psi_\mu$, and spherical coordinates centred at $y$. We get
\begin{eqnarray*}
I^{\mu, \mu'}_{i,1}&=& \int_{B_l(y_i)}\int_{\{\overline y:|y-\overline y|\leq r\}}\frac{(\widetilde \psi_{\mu}(y)-\widetilde \psi_{\mu}(\overline y))(\widetilde \psi_{\mu'}(y)-\widetilde \psi_{\mu'}(\overline y))}{|y-\overline y|^{d+1}}dy d\overline y\\
&\lesssim& |B_l(y_i)|\: \|\nabla \widetilde \psi_\mu\|_{\infty}\: \|\nabla \widetilde \psi_{\mu'}\|_{\infty}\int_0^r \frac{\rho^2\rho^{d-1}}{\rho^{d+1}}d\rho\\
&\lesssim& l^d \mu^{\frac{1}{3d+1}}\mu'^{\frac{1}{3d+1}}r=\mu^{\frac{d}{3d+1}}\mu^{\frac{1}{3d+1}}\mu'^{\frac{1}{3d+1}}r.
\end{eqnarray*}
On the other hand, using now the $L^\infty$-bound \eqref{L-infty} for $\widetilde \psi_\mu$, we deduce
\begin{eqnarray*}
I^{\mu, \mu'}_{i,2}&=& \int_{B_l(y_i)}\int_{\{\overline y:|y-\overline y|> r\}}\frac{(\widetilde \psi_{\mu}(y)-\widetilde \psi_{\mu}(\overline y))(\widetilde \psi_{\mu'}(y)-\widetilde \psi_{\mu'}(\overline y))}{|y-\overline y|^{d+1}}dy d\overline y\\
&\lesssim& |B_l(y_i)|\: \| \widetilde \psi_\mu\|_{\infty}\: \| \widetilde \psi_{\mu'}\|_{\infty}\int_r^\infty \frac{\rho^{d-1}}{\rho^{d+1}}d\rho\\
&\lesssim& l^d \mu^{\frac{2}{3d+1}}\mu'^{\frac{2}{3d+1}}\frac{1}{r}=\mu^{\frac{d}{3d+1}}\mu^{\frac{2}{3d+1}}\mu'^{\frac{2}{3d+1}}\frac{1}{r}.
\end{eqnarray*}
Optimizing the sum $I^{\mu, \mu'}_{i,1}+I^{\mu, \mu'}_{i,2}$ in $r$ , we get for $r=\mu^{\frac{1}{2(3d+1)}}\mu'^{\frac{1}{2(3d+1)}}$
\begin{eqnarray*}I^{\mu,\mu'}&\leq& 2\sum_{i=1}^N\left(I^{\mu, \mu'}_{i,1}+I^{\mu, \mu'}_{i,2}\right)\lesssim N \mu^{\frac{2d+3}{2(3d+1)}}\mu'^{\frac{3}{2(3d+1)}}\\
&\lesssim& \frac{1}{R^d} \mu^{\frac{2d+3}{2(3d+1)}}\mu'^{\frac{3}{2(3d+1)}}\int \chi_\mu dx \lesssim \mu^{\frac{2d}{3d+1}}\mu^{\frac{2d+3}{2(3d+1)}}\mu'^{\frac{3}{2(3d+1)}}\int \chi_\mu dx,
\end{eqnarray*}
where we have used the bound \eqref{N1} for $N$ and the choice $R=C_2^{1/2}\mu^{-2/(3d+1)}$.
Finally, integrating in $\mu$ and $\mu'$ we get
\begin{equation}\label{H1/2}
\begin{split}
\||\nabla|^{\frac{1}{2}}\widetilde \psi\|^2_2&=2\int_{1}^{+\infty}\frac{d\mu}{\mu}\int_1^{\mu}\frac{d\mu'}{\mu'} I^{\mu,\mu'}\\
&\lesssim \int_1^{\infty} \frac{d\mu}{\mu}\int_1^{\mu} \frac{d\mu'}{\mu'} \mu^{\frac{6d+3}{2(3d+1)}} \mu'^{\frac{3}{2(3d+1)}} \int dx \:\chi_\mu  \\
&\lesssim  \int_1^{\infty}   \mu^{\frac{3d+3}{3d+1}} \int  \chi_\mu dx \frac{d\mu}{\mu}.\end{split}\end{equation}
We now estimate the second term $\Phi\int\widetilde\psi$ appearing in \eqref{dual}. By an analogue but simpler computation, we deduce that
\begin{eqnarray*}
\int \widetilde \psi_\mu (y)dy \lesssim N \mu^{\frac{2}{3d+1}} l^d \lesssim R^{-d} \mu^{\frac{2}{3d+1}} l^d \int \chi_\mu \sim \mu^{\frac{2d}{3d+1}}\mu^{\frac{2}{3d+1}}\mu^{\frac{d}{3d+1}}\int \chi_\mu =\mu^{\frac{3d+2}{3d+1}}\int \chi_\mu.\end{eqnarray*}
Thus, plugging this into $\Phi\int\widetilde\psi$, we get
\begin{equation}\label{last}
\Phi \int \widetilde \psi =\Phi \int_1^{+\infty}\left(\int \widetilde \psi_{\mu}(y)dy\right)\frac{d\mu}{\mu} \lesssim \Phi \int_1^{+\infty} \mu^{\frac{3d+2}{3d+1}}\int \chi_\mu dx \frac{d\mu}{\mu}.
\end{equation}
Plugging \eqref{dual}, \eqref{H1/2}, and \eqref{last} into \eqref{W-H-3}, we deduce that there exists a constant $C>0$ such that
%
\begin{align}
&\int_1^{\infty}\mu^{\frac{3d+3}{3d+1}}\int \chi_\mu \frac{d\mu}{\mu}
\leq C \|\nabla u\|_1
+ CW_2^2(u,v)+\\
&\hspace{2em} +C\left(\| |\nabla|^{-\frac{1}{2}}(v-\Phi)\|^2_2\int_1^{\infty} \mu^{\frac{3d+3}{3d+1}}\int \chi_\mu dx \frac{d\mu}{\mu}\right)^{1/2}+C\Phi \int_1^{\infty} \mu^{\frac{3d+2}{3d+1}}\int \chi_\mu dx \frac{d\mu}{\mu}.
\end{align}
%
Since $(3d+2)/(3d+1)<(3d+3)/(3d+1)$, we can absorbe the last term on the right-hand side for $\Phi \leq \frac{1}{2C}$ . Moreover, using Young inequality, we can absorb the second factor of the third term on the right-hand side.
Finally we evaluate the integral in $\mu$ on the left-hand side to get
\begin{equation}\label{nu=1}
 \int(u-1)_+^{\frac{3d+3}{3d+1}} \leq 2C\left(\|\nabla u\|_1 +W_2^2(u,v)+ \||\nabla|^{-\frac{1}{2}}(v-\Phi)\|^2_2\right).\end{equation}

\textbf{Step 3.}
In this last step we show that, by scaling, the following inequality in additive form containing the parameter $\nu$ holds for $\Phi \leq \nu^{\frac{3d+1}{3d+3}}/(2C)$
\begin{equation}\label{nu+}
\int \left(u-\nu^{\frac{3d+1}{3d+3}}\right)_+^{\frac{3d+3}{3d+1}}\leq 2C\left(\|\nabla u\|_1+ \nu^{\frac{2}{d+1}} W_2^2(u,v)+\nu^{\frac{1-d}{d+1}}\||\nabla|^{-\frac{1}{2}}(v-\Phi)\|_2^2\right).\end{equation}
Indeed, set $u=M\hat u,$ $v=M\hat v,$ $\Phi=M\hat\Phi$, and $x=L \hat x$, by \eqref{nu=1} we have for $\hat\Phi \leq (2CM)^{-1}$
\begin{equation*}
L^d M^{\frac{3d+3}{3d+1}}\int\left(\hat u-M^{-1}\right)_+^{\frac{3d+3}{3d+1}} \lesssim L^{d-1} M \int |\hat \nabla \hat u| + L^{d+2}M W_2^2(\hat u, \hat v)+L^{d+1}M^2\||\nabla|^{-\frac{1}{2}}(\hat v-\hat \Phi)\|^2_2.\end{equation*}
We divide by $L^d M^{\frac{3d+3}{3d+1}}$ to get
$$\int\left(\hat u-M^{-1}\right)_+^{\frac{3d+3}{3d+1}} \lesssim L^{-1} M^{-\frac{2}{3d+1}} \int |\hat \nabla \hat u| + L^{2}M^{-\frac{2}{3d+1}} W_2^2(\hat u, \hat v)+L M^{\frac{3d-1}{3d+1}}\||\nabla|^{-\frac{1}{2}}(\hat v-\hat \Phi)\|^2_2.$$ Choosing $L=M^{-\frac{2}{3d+1}}$ we have
$$\int \left(\hat u-M^{-1}\right)_+^{\frac{3d+3}{3d+1}} \lesssim   \int |\hat \nabla \hat u| + M^{-\frac{6}{3d+1}} W_2^2(\hat u, \hat v)+M^{\frac{3d-3}{3d+1}}\||\nabla|^{-\frac{1}{2}}(\hat v-\hat \Phi)\|^2_2.$$
This implies \eqref{nu+} for $\nu=M^{-\frac{3d+3}{3d+1}}$.
\end{proof}
Finally, we prove Proposition \ref{cor}, which follows easily by Proposition \ref{strong-W-H} by scaling arguments.
\begin{proof}[Proof of Proposition \ref{cor}]
As a corollary of Proposition \ref{strong-W-H}, we deduce that given two nonnegative functions $u,v:\re^d \rightarrow \re$ with $\int u=\int v <\infty$, the following inequality in additive form holds
$$\int \left(u-\nu^{\frac{3d+1}{3d+3}}\right)_+^{\frac{3d+3}{3d+1}}\lesssim \|\nabla u\|_{1}+\nu^{\frac{2}{d+1}}W_2^2(u,v)+\nu^{\frac{1-d}{d+1}}\||\nabla|^{-\frac{1}{2}}v\|_2^2.$$
We pass from the additive form to the multiplicative one, by scaling in $x$. Indeed for $x=L\hat x$, we have
\begin{equation*}
\begin{split}
&\int \left(u-\nu^{\frac{3d+1}{3d+3}}\right)_+^{\frac{3d+3}{3d+1}}\lesssim L^{-1}\|\hat \nabla u\|_1 \\
& \hspace{1em}+ L^{\frac{2d}{d+1}}\left( (\nu L)^{\frac{2}{d+1}}W_2^2(u,v)+(\nu L)^{\frac{1-d}{d+1}} \||\nabla|^{-\frac{1}{2}}v\|_2^2\right).\end{split}\end{equation*}
Setting $\hat \nu=\nu L$ and choosing $L=\left(\|\hat \nabla u\|_1\right)^{\frac{d+1}{3d+1}}\left(\hat\nu ^{\frac{2}{d+1}}W_2^2(u,v)+\hat\nu ^{\frac{1-d}{d+1}} \||\nabla|^{-\frac{1}{2}}v\|_2^2\right)^{-\frac{d+1}{3d+1}}$, we deduce
\begin{equation*}
\int \left(u-\nu^{\frac{3d+1}{3d+3}}\right)_+^{\frac{3d+3}{3d+1}}\lesssim \left(\|\hat \nabla u\|_1\right)^{\frac{2d}{3d+1}} \left(
 \hat\nu^{\frac{2}{d+1}}W_2^2(u,v)+\hat\nu^{\frac{1-d}{d+1}}\||\nabla|^{-\frac{1}{2}}v\|_2^2\right)^{\frac{d+1}{3d+1}}.\end{equation*}
 Taking the supremum in $\nu$ (that is the supremum in $\hat \nu$ on the right-hand side, since $\hat \nu=L\nu$) and raising to the power $(3d+1)/(3d+3)$ we conclude the proof.
\end{proof}

\noindent
\textbf{Acknowledgement}: We would like to thank M. Ledoux for useful discussions on the topic of this paper.

\end{document}